\documentclass[12pt,twoside,english]{article}
\usepackage{babel,amssymb,amsthm}
\usepackage{graphicx}
\usepackage{amsmath}
\usepackage{bm}
\usepackage{float, framed}
\usepackage{wrapfig}
\usepackage[show]{ed}
\usepackage{multicol}
\usepackage{changes} 
\usepackage{subfig}

\newcommand{\bs}{\boldsymbol}
\newcommand{\bff}{\mathbf}

\newcommand{\trace}{\mathrm{tr}\,}

\newcommand{\R}{{\mathbb R}}
\newcommand{\C}{{\mathbb C}}

\newcommand{\Vh}{{\mathbb V_h}}
\newcommand{\Wh}{{\mathbf W_h}}
\newcommand{\Mh}{{\mathbf M_h}}
\newcommand{\Mhh}[1]{{\mathbf M_h^{#1}}}

\newcommand{\PM}{{\mathbf P_M}}
\newcommand{\PV}{{\Pi_V}}
\newcommand{\PW}{{\Pi_W}}
\newcommand{\Th}{{\mathcal T_h}}

\newcommand{\av}[1]{[ #1 ]_h}

\newcommand{\esig}{{e_h^\sigma}}
\newcommand{\eu}{{\mathbf e_h^u}}
\newcommand{\euhat}{{\widehat{\mathbf e}_h^u}}
\newcommand{\epsig}{{\varepsilon_h^\sigma}}
\newcommand{\epu}{{\bs\varepsilon_h^u}}

\newtheorem{proposition}{Proposition}[section]

\newtheorem{lemma}[proposition]{Lemma}
\newtheorem{theorem}[proposition]{Theorem}

\numberwithin{equation}{section}

\title{HDG methods for elastodynamics}

\date{\today}

\author{Allan Hungria${}^{(1)}$, Daniele Prada${}^{(2)}$,
 \& Francisco--Javier Sayas${}^{(1)}$\footnote{Part of this work was developed while AH and FJS partially funded by NSF grant DMS 1216356.}  \\
 \small
(1) Dept of Mathematical Sciences, University of Delaware, USA\\
\small
(2) Dept of Mathematical Sciences, Indiana University-Purdue University Indianapolis, USA \\
{\tt allanh@udel.edu, dprada@umail.iu.edu, fjsayas@udel.edu }}

\begin{document}

\maketitle

\begin{center}
Dedicated to Peter Monk on his 60th birthday
\end{center}

\begin{abstract}
We derive and analyze a hybridizable discontinuous Galerkin (HDG) method for approximating weak solutions to the equations of time-harmonic linear elasticity on a bounded Lipschitz domain in three dimensions.  The real symmetry of the stress tensor is strongly enforced and its coefficients as well as those of the displacement vector field are approximated simultaneously at optimal convergence with respect to the choice of approximating spaces, wavenumber, and mesh size.  Sufficient conditions are given so that the system is indeed transferable onto a global hybrid variable that, for larger polynomial degrees, may be approximated via a smaller-dimensional space than the original variables.  We construct several variants of this method and discuss their advantages and disadvantages, and give a systematic approach to the error analysis for these methods.   We touch briefly on the application of this error analysis to the time-dependent problem, and finally, we examine two different implementations of the method over various polynomial degrees and numerically demonstrate the convergence properties proven herein.   
\end{abstract}

\section{Introduction}

We are concerned with numerical methods for the evolution of elastic waves on general (heterogeneous anisotropic) linearly elastic solids. It is well known that elastodynamics, in the time and frequency domains, has multiple applications in the fields of geophysics, material science, structural engineering, oil exploration, aerospace, etc.  This work is a first contribution on the use of the Hybridizable Discontinuous Galerkin (HDG) method to the three-dimensional linear elastic wave equation in the time-harmonic regime, including some insights on energy conservation properties when these ideas are applied to evolutionary cases.

Mathematical literature contains a plethora of numerical methods for dealing with the elastic wave equation, each with its own virtue and applications: spectral elements \cite{Cohen:2001}, particle-based methods such as the Hamiltonian Particle method (HPM) \cite{GoMaMiTa:2012}, as well as the more finite element styled Continuous Galerkin (CG) methods \cite{Joly10:2008} and Discontinuous Galerkin (DG) methods \cite{CoNgPe:2011}.  CG is well-known for its accuracy and reliability with smooth data and simple meshes.  The DG framework is praised for its capacity to handle all sorts of complicated meshes and discontinuous data, to provide high-order accurate solutions, to perform $h/p$ adaptivity, and to retain very good scalability properties. However, DG methods have been criticized because, for the same mesh and the same polynomial degree, the number of globally coupled degrees of freedom is much larger than those of CG methods.

Certain DG methods, however, including the ones we shall explore here, have the key property of being \textit{hybridizable}, i.e., the global system can be recast in terms of (statically condensed onto) a single ``hybrid" variable that represents the trace of the solution on the boundaries of the elements.  These form a family of methods that are, naturally, called the Hybridizable Discontinuous Galerkin (HDG) methods \cite{CoGoLa:2009}.  The main idea for devising these methods is to: (i) use a characterization of the exact solution in terms of solutions of local problems and transmission conditions; (ii) use discontinuous approximations for both the solution inside each element and its trace on the element boundary; (iii) define the local solvers with the DG method; (iv) define a global problem by weakly imposing the transmission conditions.  This creates a global linear system for only the hybrid variable, which is solved, after which the unknowns are recovered locally, again in parallel.  This is similar to the hybridized implementation of mixed methods such as the Raviart-Thomas elements (see \cite{CoGoLa:2009}, \cite{MoScSi:2010} and \cite{Sayas:2013} for more on this), except that the HDG method uses different polynomial spaces and a stabilization function instead of a stable mixed finite element pair.  In certain cases, the hybrid space is smaller than that of the displacement/stress spaces \cite{KiMoShYa2015}, and this has resulted in renewed interest in HDG.

The HDG methodology was successfully applied to time-harmonic acoustic waves by Roland Griesmaier and Peter Monk \cite{GrMo:2011}.  Their analysis involves first rephrasing the classical system as first order in frequency before moving to the weak formulation.  Testing the equations with the projected errors leads to a G{\aa}rding-type identity, and, combined with the dual equations to the classical system, the projected errors of both amplitude and its gradient can be bounded.  This last bound requires a rather involved bootstrapping argument which is indispensable within our argument here.   

Hybridizable DG methods have lately enjoyed further exposure in time-domain wave problems.  For example, Cockburn and Quenneville-B{\'e}lair's work on HDG for the acoustic wave equation \cite{CoQu:2014} provides much of the framework for the insights on the time-domain elastic problem at the end of this work.  Nguyen, Peraire, and Cockburn implement an implicit HDG numerical scheme for both time-dependent acoustic and elastic equations \cite{CoNgPe:2011}, and more recently, Stanglmeier, Nguyen, Peraire, and Cockburn explore an explicit scheme for the acoustic case \cite{CoNgPeSt:2016}.

The vector field formulation of elasticity introduces several distinct complications in both the analysis and the implementation of HDG.  Cockburn, Soon, and Stolarski give a numerical implementation of HDG for planar elasticity, along with a proof of existence and uniqueness of a solution to their particular HDG formulation \cite{CoSoSt:2009}.  Cockburn, Fu, and Stolarski go on to analyze the convergence of this last method, which uses degree $k$ polynomial bases for displacement, stress, and hybrid spaces.  They prove convergence at an order of $k+1$ for displacement and $k+1/2$ for the stress \cite{CoFuSt:2014}, which is suboptimal; this has prompted the exploration of optimally convergent HDG methods.

One issue is that the tailored HDG projections often used in the analysis may not play well with the symmetry of the approximate stress tensor.  Another is that using bases of the same polynomial degree for displacement, stress, and hybrid spaces leads to a suboptimal method.  One method for addressing both of these issues is to introduce special divergence-free symmetric ``bubble matrices" as in \cite{CoGoGu:2010}, providing an extra control on the stress-associated approximation space.  This yields a weakly-enforced symmetry of the approximate stress as well as optimal convergence of a postprocessed solution, taking advantage of some superconvergent quantities.

Another approach entirely is that of Weifeng Qiu, Jiguang Shen, and Ke Shi for the steady-state elasticity problem \cite{QiSh:2014}.  The special tailored HDG projections are left behind for simpler $L^2$ projections, and the displacement-associated approximation space is expanded by one polynomial degree.  While this does then require some extra terms to be bounded in the analysis, the net result is shown to achieve optimal convergence directly. (Note that due to the disparity of polynomial degrees for the stress and displacement, optimal convergence of this method yields the same quality of the solution as a postprocessed method based on a superconvergent scheme.)  An important feature of this approach is the strong symmetry of the approximate stress.  See the introduction of \cite{QiSh:2014} for more on this.  Expanding the polynomial degree of the primal unknown by one is an idea that can be traced back to Lehrenfeld and Schoberl \cite{LeSc:2015}, but Qiu, Shen and Shi compensate by adjusting the order of the primal unknown piece of the stabilization function to $O(h^{-1})$ as well as a projection operator from primal approximation space onto hybrid space.

Our choice of polynomial approximating spaces and projections is that of Qiu, Shen, and Shi in order to be able to work on the most general polyhedral mesh possible.  However, the frequency-domain problem, unlike the steady-state problem, is not coercive, so we wind up with a G\aa rding-type identity similar to that of Griesmaier and Monk's \cite{GrMo:2011}, after following their example and first phrasing the classical system as first order in both frequency and space.  The two analytical recipes from \cite{QiSh:2014} and \cite{GrMo:2011} are here carefully blended to approach the time-harmonic elasticity case, which has implications on the choice of numerical flux and its dependence on the wavenumber.  

The following treatment of HDG for time-harmonic elasticity, however, comes with its own complications, not only with regard to the hybridization of the DG scheme, but also in consideration of the dependence on the wavenumber.  We have also developed a simplified system for dealing with the double-bootstrapping process, which is now even messier due to the use of $L^2$ projections rather than tailored HDG projections.  By varying the numerical flux, we wind up with several different HDG methods for the time-harmonic linear elastic problem.  We proceed to show how some of these methods can be used to produce semi-discretizations in the time domain and that one of them is actually conservative.

What follows is a rigorous treatment of the error analysis and well-posedness of HDG methodology as applied to the problem of three-dimensional time-harmonic elasticity on a polyhedron with mixed boundary conditions and strong symmetric stresses.  We explore how this analysis can shape the stability mechanism for a method of numerically integrating the time-dependent system, in particular for the second-order-in-frequency case.  Numerical experiments are carried out to demonstrate convergence of both the first-order method and a second-order variant.  We then compare, using various polynomial degrees and tetrahedrizations, the sizes of the global linear systems involved in HDG and Lagrange element CG, demonstrating an advantage of HDG at large polynomial degrees.

\paragraph{Notation.} Given an open set $\Omega$, we will write
\[
(u,v)_\Omega:=\int_\Omega u\, v,
	\quad
(\mathbf u,\mathbf v)_\Omega:=\int_\Omega \mathbf u\cdot\mathbf v,
	\quad
(\bs\xi,\bs\chi)_\Omega:=\int_\Omega \bs\xi:\bs\chi:=\int_\Omega \mathrm{trace}(\bs\chi^\top \bs\xi),
\]
for real square-integrable scalar, vector-valued, and matrix-valued functions. The symbol $\top$ will be used for real transposition of matrices. When used for complex-valued fields, all brackets will still be defined in the same way, and will be therefore bilinear and not sesquilinear. In the same spirit, $\top$ will denote transposition and the colon will be defined as above, even when applied to complex matrices. 
The set of symmetric real $3\times 3$ matrices will be denoted $\R^{3\times 3}_{\mathrm{sym}}$ and the set of symmetric (not Hermitian) complex $3\times 3$ matrices will be denoted $\C^{3\times 3}_{\mathrm{sym}}$.  Similarly, the set of antisymmetric real $3\times 3$ matrices will be denoted $\R^{3\times 3}_{\mathrm{skw}}$ and the set of antisymmetric complex $3\times 3$ matrices will be denoted $\C^{3\times 3}_{\mathrm{skw}}$.

\section{Problem setting and HDG discretization}\label{sec:2}

Let $\Omega\subset \mathbb R^3$ be a polyhedron with Lipschitz boundary $\Gamma$. We assume that $\Gamma$ is divided into two non-overlapping parts $\Gamma_D$ and $\Gamma_N$, where we will respectively impose Dirichlet and Neumann boundary conditions. For simplicity we will assume that each of $\Gamma_D$ and $\Gamma_N$ are made up of full faces of $\Gamma$. We will be looking for a displacement field $\mathbf u:\Omega\to \C^3$ and for the associated stress tensor $\widetilde{\bs\sigma}:\Omega\to \C^{3\times 3}_{\mathrm{sym}}$. The stress field is given by a general linear non-homogeneous anisotropic law:
\[
\widetilde{\bs\sigma}=\mathcal C\bs\varepsilon(\mathbf u),
\qquad\bs\varepsilon(\mathbf u):=\tfrac12 (\nabla\mathbf u+(\nabla \mathbf u)^\top),
\]
where for almost every $\mathbf x\in \Omega$, the linear operator $\mathcal C(\mathbf x)$ transforms real symmetric matrices into real symmetric matrices, satisfies the symmetry condition
\[
(\mathcal C(\mathbf x)\bs\xi):\bs\chi 
=\bs\xi : (\mathcal C(\mathbf x)\bs\chi) \quad \forall \bs\xi,\bs\chi\in \mathbb R^{3\times 3}_{\mathrm{sym}}
\]
and the positivity condition
\[
(\mathcal C(\mathbf x) \bs\xi):\bs\xi
\ge C_0\,  \bs\xi:\bs\xi \qquad \forall \bs\xi\in \R^{3\times 3}_{\mathrm{sym}},
\]
for some $C_0>0$.
Moreover, we assume that the components of the tensor $\mathcal C$ with respect to the canonical basis of $\mathbb R^{3\times 3}_{\mathrm{sym}}$ are $L^\infty(\Omega)$ functions. The other physical parameter in the equations to follow is the strictly positive bounded density $\rho:\Omega \to \R$, so that the weighted norm is equivalent 
\[
C_1 \|\mathbf u\|_\Omega^2 \le \|\mathbf u\|_\rho^2:=(\rho\mathbf u,\overline{\mathbf u})_\Omega
\le C_2 \|\mathbf u\|_\Omega^2.
\]
In strong primal form, our problem is the search of $\mathbf u$ such that
\begin{subequations}\label{eq:2.1}
\begin{alignat}{6}
\nabla\cdot\widetilde{\bs\sigma} + \kappa^2\rho\,\mathbf u
	 &= \widetilde{\mathbf f} &\qquad &\mbox{in $\Omega$},\\
\mathbf u &=\mathbf g_D & & \mbox{on $\Gamma_D$},\\
\widetilde{\bs\sigma}\,\mathbf n &=\widetilde{\mathbf g}_N & & \mbox{on $\Gamma_N$},
\end{alignat}
\end{subequations}
where: $\widetilde{\bs\sigma}=\mathcal C\bs\varepsilon(\mathbf u)$, the divergence operator is applied to the rows of $\widetilde{\bs\sigma}$, $\widetilde{\mathbf f}\in \mathbf L^2(\Omega)=L^2(\Omega)^3$, $\mathbf g_D\in \mathbf H^{1/2}(\Gamma_D)=H^{1/2}(\Gamma_D)^3$, $\widetilde{\mathbf g}_N \in \mathbf L^2(\Gamma_N)=L^2(\Gamma_N)^3$, $\mathbf n$ is the unit outward-pointing normal vector field on $\Gamma_N$,  and $\kappa > 0$ is the wave number. In the weak primal formulation (where problem \eqref{eq:2.1} is typically studied), the Dirichlet condition is understood in the sense of traces, while the Neumann condition holds in a dual space with negative Sobolev index on the boundary $\Gamma_N$. We will assume that $\kappa^2$ is not an eigenvalue for the associated Navier-Lam\'e operator $\mathbf u\mapsto -\nabla\cdot (\mathcal C\bs\varepsilon(\mathbf u))$ with the given boundary conditions, i.e., we assume that the only solution of \eqref{eq:2.1} with zero right-hand side is the trivial solution. 

The discretization will be done for a first order (in space and frequency) reformulation of \eqref{eq:2.1}. We first need to invert the elastic law $\mathcal C$. With the hypothesis given for $\mathcal C$, we can assert that for almost every $\mathbf x\in \Omega$, there exists a linear operator $\mathcal A(\mathbf x)=\mathcal C(\mathbf x)^{-1}$, transforming real symmetric matrices into real symmetric matrices. On the set of matrix-valued $\bs\xi:\Omega \to \mathbb C^{3\times 3}_{\mathrm{sym}}$ functions with $L^2(\Omega)$ components, we define the elastic potential norm
\[
C_1\| \bs\xi\|_\Omega^2 \le \| \bs\xi\|_{\mathcal A}^2 :=
(\mathcal A \bs\xi,\overline{\bs\xi})_{\Omega}=(\mathcal A\bs\xi_{\mathrm{re}}+\imath \mathcal A\bs\xi_{\mathrm{im}},\overline{\bs\xi})_\Omega
\le C_2 \| \bs\xi\|_\Omega^2.
\]
We emphasize that we will always use bilinear (not sesquilinear) brackets for $L^2$-type products, and that the symbol used for the Frobenius product of matrices (the colon) will not include conjugation. We then introduce the new unknown and data
\[
\bs\sigma:=\frac\imath\kappa \widetilde{\bs\sigma}=\frac\imath\kappa \mathcal C \bs\varepsilon(\mathbf u),
	\qquad
\mathbf f:=\frac\imath\kappa \widetilde{\mathbf f},
	\qquad
\mathbf g_N:=\frac\imath\kappa \widetilde{\mathbf g}_N,
\]
and write \eqref{eq:2.1} as the equivalent first order system
\begin{subequations}\label{eq:2.2}
\begin{alignat}{6}
\imath\kappa\mathcal A\bs\sigma +\bs\varepsilon(\mathbf u) &=\mathbf 0
	&\qquad &\mbox{in $\Omega$},\\
\nabla\cdot\bs\sigma +\imath\kappa\,\rho\,\mathbf u &=\mathbf f
	&\qquad &\mbox{in $\Omega$},\\
\mathbf u &=\mathbf g_D
	&\qquad &\mbox{on $\Gamma_D$},\\
\bs\sigma\mathbf n &=\mathbf g_N
	&\qquad &\mbox{on $\Gamma_N$}.
\end{alignat}
\end{subequations}
A similar formulation can be found using the original data and the stress tensor $\widetilde{\bf\sigma}$, so that the equations are $\mathcal A\widetilde{\bs\sigma}-\bs\varepsilon(\mathbf u)=\mathbf 0$ and $\nabla\cdot\widetilde{\bs\sigma}+\kappa^2 \,\rho\,\mathbf u=\widetilde{\mathbf f}$. (This will be discussed in Section \ref{sec:6}.)

We next introduce the HDG discretization of \eqref{eq:2.2}. Since the method we use is Qiu \& Shi's \cite{QiSh:2014}, we will not repeat the derivation. We start with a shape-regular conforming tetrahedrization $\Th$ of the domain $\Omega$. The set of all faces of elements of $\Th$ is denoted $\mathcal E_h$, and we will sometimes understand that $\mathcal E_h$ is the geometric skeleton of the triangulation, i.e., the union of all the faces of all elements. The method involves three discrete spaces
\begin{subequations}\label{eq:2.3}
\begin{eqnarray}
\Vh &:=& \{ \bs\xi:\Omega\to \C^{3\times 3}_{\mathrm{sym}}\,:\, \bs\xi|_K \in \mathcal P_k(K;\C^{3\times 3}_{\mathrm{sym}}) \quad \forall K\in \Th\},\\
\Wh &:=& \{ \mathbf u:\Omega\to \C^3\,:\, \mathbf u|_K \in \mathcal P_{k+1}(K;\C^3)\quad
\forall K\in \Th\},\\
\Mh &:=& \{ \bs\mu:\mathcal E_h \to \C^3\,:\, \bs\mu|_F\in \mathcal P_k(F;\C^3)\quad
\forall K\in \mathcal E_h\}.
\end{eqnarray}
\end{subequations}
In \eqref{eq:2.3}, $\mathcal P_r(K;S)$ is the set of polynomials of total degree up to $r$ defined on $K$ and with values in $S\in \{\C^{3\times 3}_{\mathrm{sym}},\C^3\}$, while $\mathcal P_k(F;\C^3)$ are vector valued polynomials on the tangential coordinates  defined on the face $F$ and of degree not greater than $k$. We will also use the orthogonal projector
\begin{equation}\label{eq:2.4new}
\PM : \prod_{K\in \Th} L^2(\partial K) \longrightarrow \prod_{K\in \Th} \prod_{F\in \mathcal E(K)} \mathcal P_k(F;\C^3),
\end{equation}
where $\mathcal E(K)$ is the set of faces of $\partial K$. Note that $\Mh$ can be identified with the subspace of the set of the right-hand side of \eqref{eq:2.4new} consisting of functions that are single-valued on internal faces.

Stabilization is carried out through a function $\bs\tau$ defined as follows: for each element $K\in \Th$, a function $\bs\tau_K:\partial K \to \mathbb R^{3\times 3}_{\mathrm{sym}}$ satisfying (a) $\bs\tau_K|_{F}$ is constant on each $F\in \mathcal E(K)$; (b) there exist two positive constants such that
\begin{equation}\label{eq:tau}
C_1 h_K^{-1} \|\bs\mu\|_{\partial K}^2 \le \langle\bs\tau_K \bs\mu,\overline{\bs\mu}\rangle_{\partial K}
	\le C_2 h_K^{-1} \|\bs\mu\|_{\partial K}^2
	\quad \forall \bs\mu\in \mathbf L^2(\partial K), \quad \forall K\in \Th,
\end{equation}
where $h_K$ is the diameter of $K$. The symbol $\bs\tau$ will be used to denote the function defined on the set of boundaries of all elements as above, understanding that $\bs\tau$ can be double-valued on interior faces. The numerical fluxes follow the pattern of HDG methods: the one corresponding to the displacement will be an unknown $\widehat{\mathbf u}_h \in \Mh$, while the one related to the (normal) stress is given elementwise with a formula in terms of all the unknowns 
\begin{equation}\label{eq:2.4}
\widehat{\bs\sigma}_h\mathbf n:=\bs\sigma_h\mathbf n+\bs\tau_K (\PM \mathbf u_h-\widehat{\mathbf u}_h) 
\,:\, \partial K\to \C^3.
\end{equation}
Here the normal vector field $\mathbf n:\partial K\to \R^3$ is unitary and points to the exterior of $K$. At this time, we can write the HDG discrete equations for \eqref{eq:2.2}. We look for $(\bs\sigma_h,\mathbf u_h,\widehat{\mathbf u}_h)\in \Vh\times \Wh\times \Mh$ satisfying
\begin{subequations}\label{eq:2.5}
\begin{alignat}{6}
\label{eq:2.5a}
\imath\kappa (\mathcal A\bs\sigma_h,\bs\xi)_{\Th}
	-(\mathbf u_h,\nabla\cdot\bs\xi)_\Th+\langle\widehat{\mathbf u}_h,\bs\xi\mathbf n\rangle_{\partial\Th}
	 & = 0 & \qquad & \forall \bs\xi\in \Vh,\\
\label{eq:2.5b}	
-(\bs\sigma_h,\nabla\mathbf w)_{\Th}
	+\langle \widehat{\bs\sigma}_h\mathbf n,\mathbf w\rangle_{\partial\Th}
	+\imath\kappa (\rho\,\mathbf u_h,\mathbf w)_\Th &=(\mathbf f,\mathbf w)_\Th
		& & \forall \mathbf w\in \Wh,\\
\label{eq:2.5c}
\langle\widehat{\bs\sigma}_h\mathbf n,\bs\mu\rangle_{\partial\Th\setminus\Gamma_D}
	&=\langle\mathbf g_N,\bs\mu\rangle_{\Gamma_N} & & \forall \bs\mu \in \Mh,\\
\label{eq:2.5d}
\langle\widehat{\mathbf u}_h,\bs\mu\rangle_{\Gamma_D} &=\langle\mathbf g_D,\bs\mu\rangle_{\Gamma_D}
	& &\forall \bs\mu\in \Mh,
\end{alignat}
\end{subequations}
with \eqref{eq:2.4} as the definition of $\widehat{\bs\sigma}_h\mathbf n$ and brackets defined as follows:
\[
(\mathbf u,\mathbf v)_\Th:=\sum_{K\in \Th} (\mathbf u,\mathbf v)_K, \qquad
\langle \mathbf u,\mathbf v\rangle_{\partial\Th}:=\sum_{K\in \Th} \langle \mathbf u,\mathbf v\rangle_{\partial K}
	:=\sum_{K\in \Th}\int_{\partial K}\mathbf u\cdot\mathbf v,
\]
and
\[
\langle \mathbf u,\mathbf v\rangle_{\partial\Th\setminus\Gamma_D}
:=\sum_{K\in \Th} \langle \mathbf u,\mathbf v\rangle_{\partial K\setminus\Gamma_D}.
\]
Equations \eqref{eq:2.5c} and \eqref{eq:2.5d} can be added together as a single equation tested against $\Mh$, which shows that \eqref{eq:2.5} is a square system of linear equations. The discrete momentum equation \eqref{eq:2.5b} can be equivalently written as
\begin{equation}\label{eq:2.5bnew}
(\nabla\cdot\bs\sigma_h,\mathbf w)_\Th
	+\langle\bs\tau (\PM\mathbf u_h-\widehat{\mathbf u}_h),\PM \mathbf w\rangle_{\partial\Th}
	+\imath\kappa (\rho\,\mathbf u_h,\mathbf w)_\Th
	=(\mathbf f,\mathbf w)_\Th\quad \forall\mathbf w\in \Wh.
\end{equation}
We note that the degree of the polynomial space used for $\mathbf u_h$ is one higher than the one used for the other unknowns and the fact that $\PM$ has been introduced in the definition of the flux \eqref{eq:2.4} so that $\widehat{\bs\sigma}_h\mathbf n \in 
\prod_{F\in \mathcal E(K)} \mathcal P_k(F;\C^3).$

\section{Main results}\label{sec:3}

\paragraph{Regularity assumptions.}
From now on we will assume that $\rho$ and the coefficients of $\mathcal C$ are in $W^{1,\infty}(\Th)$. 
Let us now consider the coercive problem
\begin{subequations}\label{eq:3.1}
\begin{alignat}{6}
\nabla\cdot(\mathcal C \bs\varepsilon(\mathbf w))-\rho\,\mathbf w
	 &= \mathbf r &\qquad &\mbox{in $\Omega$},\\
\mathbf w &=\bs 0 & & \mbox{on $\Gamma_D$},\\
(\mathcal C \bs\varepsilon(\mathbf w))\,\mathbf n 
&=\bs 0 & & \mbox{on $\Gamma_N$}.
\end{alignat}
\end{subequations}
We will also assume that the solution of \eqref{eq:3.1} for arbitrary $\mathbf r\in L^2(\Omega;\mathbb R^3)$ is in $H^2(\Omega;\mathbb R^3)$ and that there exists a constant $C>0$ such that
\begin{equation}\label{eq:3.2}
\| \mathbf w\|_{2,\Omega} \le C \|\mathbf r\|_\Omega.
\end{equation}
For the time-harmonic problem, we will denote by $C_\kappa>0$ the constant such that the solution of
\begin{subequations}\label{eq:3.3}
\begin{alignat}{6}
\nabla\cdot(\mathcal C \bs\varepsilon(\mathbf w))+\kappa^2\rho\,\mathbf w
	 &= \mathbf r &\qquad &\mbox{in $\Omega$},\\
\mathbf w &=\bs 0 & & \mbox{on $\Gamma_D$},\\
(\mathcal C \bs\varepsilon(\mathbf w))\,\mathbf n 
&=\bs 0 & & \mbox{on $\Gamma_N$}.
\end{alignat}
\end{subequations}
can be bounded by
\begin{equation}\label{eq:3.4}
\| \mathbf w\|_{1,\Omega}\le C_\kappa \|\mathbf r\|_\Omega.
\end{equation}
Note that we have assumed the unique solvability of \eqref{eq:3.3}.

\paragraph{Error quantities.}
The error analysis will be carried out by comparing numerical solutions and orthogonal projections. Let $\PV: L^2(\Omega;\C^{3\times 3}_{\mathrm{sym}})\to \Vh$ and $\PW:L^2(\Omega;\C^3)\to\Wh$ be the orthogonal projections onto the discrete spaces. Consider the errors
\[
\esig:=\PV \bs\sigma-\bs\sigma_h,
	\qquad
\eu:=\PW \mathbf u-\mathbf u_h,
	\qquad
\euhat:=\PM \mathbf u-\widehat{\mathbf u}_h,
\]
and the best approximation errors
\[
\epsig:=\PV\bs\sigma-\bs\sigma,
	\qquad
\epu:=\PW \mathbf u-\mathbf u.
\]
For convenience, we introduce the skeleton norm
\[
\| \bs\mu\|_\tau:=\langle \bs\tau \bs\mu,\overline{\bs\mu}\rangle_{\partial\Th}^{1/2}.
\]
For the error analysis we will allow constants to depend on the density $\rho$ and on the coefficients of $\mathcal A$. While the influence of these physical coefficients in the inequalities can be tracked with careful arguments, the results seem to be too involved to obtain precise conclusions on how $h$ and $\kappa$ interact with them.  However, we will pay attention to the maximum spectral value of the inverse compliance tensor, i.e., to the positive bounded function such that for almost every $\mathbf x\in \Omega$
\begin{equation}\label{eq:5.3new}
(\mathcal A(\mathbf x)\bs\xi):\bs\xi\le c_{\mathcal A}(\mathbf x) \, 
	 \bs\xi :\bs\xi  \qquad \forall \bs\xi \in \R^{3\times 3}_{\mathrm{sym}}.
\end{equation}

\begin{theorem}\label{the:3.1}
There exist $C_1, C_2>0$, dependent only on the shape-regularity of $\Th$, the density $\rho$ and the coefficients of the inverse compliance tensor $\mathcal A$ such that
if $h (1+\kappa)^{3/2}(1+\kappa C_\kappa+C_\kappa)$ is small enough, then the errors can be bounded by
\[
\| \esig\|_{\mathcal A}+\kappa^{-1/2} \| \PM \eu-\euhat\|_\tau
	\le C_1 (1+\kappa^{-1/2}) \big( h^t |\bs\sigma|_{t,\Omega}+h^{s-1} |\mathbf u|_{s,\Omega}\big)
\]
and
\[
\| \eu\|_\Omega\le C_2 (1+\kappa C_\kappa) \kappa^{-1/2}(1+\kappa)^2 
\big( h^{t+1} |\bs\sigma|_{t,\Omega}+h^{s} |\mathbf u|_{s,\Omega}\big),
\]
if $k\ge 1$, $\mathbf u\in H^s(\Omega;\C^3)$ with $1\le s\le k+2$, and $\bs\sigma\in H^t(\Omega;\C^{3\times 3})$ with $1\le t\le k+1$.
\end{theorem}

Optimal error estimates are
\[
\| \bs\sigma-\bs\sigma_h\|_\Omega=\mathcal O(h^{k+1}), 
	\qquad
\|  \mathbf u-\mathbf u_h\|_\Omega=\mathcal O(h^{k+2}).
\]
With some additional scaling inequalities, keeping in mind that $\bs\tau$ scales like $h^{-1}$ elementwise, it is possible to show that
\[
\| \mathbf u-\widehat{\mathbf u}_h\|_\tau =\mathcal O(h^{k+1}).
\]
The estimates of Theorem \ref{the:3.1} can also be written in terms of the original physical variables. If we denote $\widetilde{\bs\sigma}_h:=-\imath\kappa\bs\sigma_h$, then
\begin{alignat*}{6}
\| \PV \widetilde{\bs\sigma}-\widetilde{\bs\sigma}_h\|_\Omega+
	\kappa^{1/2} \| \PM \eu-\euhat\|_\tau 
	& \le
	C_1 (1+\kappa^{-1/2}) \big( h^t |\widetilde{\bs\sigma}|_{t,\Omega}+h^{s-1}\kappa |\mathbf u|_{s,\Omega}\big),\\
\| \eu\|_\Omega
	&\le C_2 (1+\kappa C_\kappa) \kappa^{-3/2}(1+\kappa)^2 
       \big( h^{t+1} |\widetilde{\bs\sigma}|_{t,\Omega}+h^{s} \kappa|\mathbf u|_{s,\Omega}\big).
\end{alignat*}

\paragraph{Unique solvability.}
Theorem \ref{the:3.1} can be used to prove existence and uniqueness of solution of \eqref{eq:2.5} for $h$ small enough (depending on the wave number $\kappa$).  The argument is as follows. Consider the system \eqref{eq:2.5} with homogeneous data: $\mathbf f=\bs 0$, $\mathbf g_N=\bs 0$, and $\mathbf g_D=\bs 0$. Let $(\bs\sigma_h,\mathbf u_h,\widehat{\mathbf u}_h)$ be any solution of this homogeneous set of linear equations. Theorem \ref{the:3.1} applied to this solution and the exact zero solution shows that $(\bs\sigma_h,\mathbf u_h,\widehat{\mathbf u}_h)$ has to vanish. Therefore, the linear system \eqref{eq:2.5} (with as many equations as unknowns) is uniquely solvable for any right-hand side. The logic of the use of Theorem \ref{the:3.1} is slightly warped: it assumes the existence of a discrete solution, which we know to happen at least for the homogeneous case, and then it uses the error estimates to show that the system is actually uniquely solvable.

\section{Local solvability and energy identity}

\begin{lemma}\label{lemma:5.1}
There exists $C>0$, depending only on the shape regularity of the grid, such that
\[
\|\mathbf v\|_K \le C h_K \|\bs\varepsilon(\mathbf v)\|_K
\]
for all $\mathbf v\in H^1(K;\C^3)$ satisfying
\begin{equation}\label{eq:5.1}
\langle \mathbf v,\bs\mu\rangle_{\partial K}=0 \qquad \forall \bs\mu\in \prod_{F\in \mathcal E(K)} \mathcal P_0(F;\C^3).
\end{equation}
\end{lemma}

\begin{proof}
A scaling argument, using only that $\int_{\partial K}\mathbf v=\mathbf 0$ and a Poincar\'e inequality on the reference element prove that
\[
\|\mathbf v\|_K \le C h_K \|\nabla\mathbf v\|_K.
\]
On the other hand, by a straightforward extension of \cite[Lemma 4.1]{QiSh:2014} to our complex-valued fields, we have the local Korn inequality
\begin{equation}\label{eq:5.2}
\inf_{\bs\xi\in \mathbb C^{3\times 3}_{\mathrm{skw}}}\| \nabla\mathbf v+\bs\xi\|_K
	\le C \|\bs\varepsilon(\mathbf v)\|_K \quad\forall\mathbf v\in H^1(K;\C^3).
\end{equation}
The constant in \eqref{eq:5.2} depends only on the shape-regularity constant of the mesh. Finally, if $\bs\xi\in \C^{3\times 3}$, then
\[
(\nabla\mathbf v,\bs\xi)_K=\langle \mathbf v,\bs\xi\mathbf n\rangle_{\partial K}=0,
\]
since $\mathbf v$ satisfies \eqref{eq:5.1}. Therefore 
\[
\inf_{\bs\xi\in \mathbb C^{3\times 3}_{\mathrm{skw}}}\| \nabla\mathbf v+\bs\xi\|_K=\|\nabla\mathbf v\|_K
\]
and the proof is finished.
\end{proof}

The following result shows that the local equations associated to \eqref{eq:2.5a}-\eqref{eq:2.5b} are uniquely solvable, i.e., given the data functions and $\widehat{\mathbf u}_h$, we can compute $\bs\sigma_h$ and $\mathbf u_h$ element by element. This is the key ingredient to show that the HDG method \eqref{eq:2.5} is actually hybridizable, that is, it can be recast as a linear system where $\widehat{\mathbf u}_h\in \Mh$ is the only variable. To simplify the proof, we introduce the weighted norms
\[
\| \bs\xi\|_{\mathcal A,K}^2:=(\mathcal A\bs\xi,\overline{\bs\xi})_K,
	\qquad
\| \mathbf v\|_{\rho,K}^2:=(\rho\,\mathbf v,\overline{\mathbf v})_K.
\]

\begin{proposition}[Local solvers]\label{prop:5.2}
If $C>0$ is the constant of Lemma \ref{lemma:5.1} and
\begin{equation}\label{eq:5.4new}
\kappa\, h_K < \frac1{C\,\| c_{\mathcal A}\|_{L^\infty(K)}^{1/2} \| \rho\|_{L^\infty(K)}^{1/2}},
\end{equation}
then the local solver associated to the element $K\in \Th$ is well defined. In other words, 
if $(\bs\sigma,\mathbf u)\in \mathcal P_k(K;\C^{3\times3}_{\mathrm{sym}})\times \mathcal P_{k+1}(K;\C^3)$ satisfies
\begin{subequations}
\begin{alignat}{6}
\label{eq:5.3a}
\imath\kappa (\mathcal A\bs\sigma,\bs\xi)_K -(\mathbf u,\nabla\cdot\bs\xi)_K 
	&=0
	&\qquad &\forall \bs\xi\in \mathcal P_k(K;\C^{3\times3}_{\mathrm{sym}}),\\
\label{eq:5.3b}
(\nabla\cdot\bs\sigma,\mathbf w)_K+\langle\bs\tau\PM\mathbf u,\mathbf w\rangle_{\partial K}
	+\imath\kappa(\rho\mathbf u,\mathbf w)_K
	&=0
	&\qquad & \forall\mathbf w\in \mathcal P_{k+1}(K;\C^3),
	\end{alignat}
\end{subequations}
then $(\bs\sigma,\mathbf u)=(\mathbf 0,\mathbf 0)$. 
\end{proposition}

\begin{proof}
Note that we only need to prove that $\mathbf u=\mathbf 0$. Testing \eqref{eq:5.3a} with $\overline{\bs\sigma}$, conjugating \eqref{eq:5.3b} and testing it with $\overline{\mathbf u}$, and adding the result of these two equations, it follows that
\[
\imath\kappa\left(\|\bs\sigma\|_{\mathcal A,K}^2-\|\mathbf u\|_{\rho,K}^2\right)
	+\langle\bs\tau\PM \overline{\mathbf u},\PM \mathbf u\rangle_{\partial K}=0.
\]
By \eqref{eq:tau} it follows that $\PM \mathbf u=\mathbf 0$ and $\|\bs\sigma\|_{\mathcal A,K}=\|\mathbf u\|_{\rho,K}$. Going back to \eqref{eq:5.3a}, integrating by parts, and using that $\PM \mathbf u=\mathbf 0$, it follows that
\begin{equation}\label{eq:5.4}
\imath\kappa (\mathcal A\bs\sigma,\bs\xi)_K+(\nabla\mathbf u,\bs\xi)_K=0 \quad \forall  \bs\xi\in \mathcal P_k(K;\C^{3\times3}_{\mathrm{sym}}).
\end{equation}
Testing \eqref{eq:5.4} with $\bs\xi=\bs\varepsilon(\overline{\mathbf u})$, it follows that
\[
\| \bs\varepsilon( \mathbf u)\|_K^2
	= (\nabla\mathbf u,\bs\varepsilon(\overline{\mathbf u}))_K
	        = \kappa | (\mathcal A\bs\sigma,\bs\varepsilon(\overline{\mathbf u}))_K|  \le 
	        \kappa \| c_{\mathcal A}\|_{L^\infty(K)}^{1/2} \|\bs\sigma\|_{\mathcal A,K} \| \bs\varepsilon( \mathbf u)\|_K,
\]
where we have used \eqref{eq:5.3new}. Note that $\mathbf u$ satisfies \eqref{eq:5.1}, given the fact that $\PM\mathbf u=\mathbf 0$. Therefore, by Lemma \ref{lemma:5.1}, if $\bs\varepsilon(\mathbf u)=\mathbf 0$, then $\mathbf u=\mathbf 0$ and the proof is finished. Otherwise $\mathbf u\neq \mathbf 0$ and, by Lemma \ref{lemma:5.1} and the equality $\|\bs\sigma\|_{\mathcal A,K}=\|\mathbf u\|_{\rho,K}$, we can bound
\begin{eqnarray*}
\| \mathbf u\|_K & \le & C h_K \| \bs\varepsilon(\mathbf u)\|_K 
			\le C \kappa\, h_K \| c_{\mathcal A}\|_{L^\infty(K)}^{1/2} \|\bs\sigma\|_{\mathcal A,K}
			= C \kappa\, h_K \| c_{\mathcal A}\|_{L^\infty(K)}^{1/2} \|\mathbf u\|_{\rho,K} \\
			&\le & C \kappa\, h_K \| c_{\mathcal A}\|_{L^\infty(K)}^{1/2} \| \rho\|_{L^\infty(K)}^{1/2}
					\|\mathbf u\|_K
\end{eqnarray*}
and we arrive at a contradiction if \eqref{eq:5.4new} holds.
\end{proof}

\begin{proposition}\label{prop:5.3}
\begin{alignat*}{6}
& \imath\kappa (\|\esig\|_{\mathcal A}^2-\|\eu\|_\rho^2)+\|\PM \eu-\euhat\|_\tau^2 \\
& \hspace{5pt} =\imath\kappa\left( (\mathcal A\epsig,\overline\esig)_\Th-(\rho \overline\epu,\eu)_\Th\right)
	+\langle\overline\epsig\mathbf n,\eu-\euhat\rangle_{\partial\Th}
	+\langle\bs\tau \, \overline\epu, \PM \eu-\euhat\rangle_{\partial\Th}.
\end{alignat*}
\end{proposition}

\begin{proof}
Substituting $(\PV\bs\sigma,\PW\mathbf u,\PM \mathbf u)$, where $(\bs\sigma,\mathbf u)$ is the exact solution of \eqref{eq:2.2} in the left-hand side of discrete equations, and subtracting the actual discrete equations \eqref{eq:2.5} (with \eqref{eq:2.5b} better written in the form \eqref{eq:2.5bnew}), it is simple to prove that the errors $(\esig,\eu,\euhat) \in \Vh\times \Wh\times \Mh$ satisfy
\begin{subequations}\label{eq:5.7}
\begin{alignat}{6}
\label{eq:5.7a}
\imath\kappa (\mathcal A \esig,\bs\xi)_\Th-(\eu,\nabla\cdot\bs\xi)_\Th+\langle \euhat,\bs\xi\mathbf n\rangle_{\partial\Th}
	&=\imath\kappa (\mathcal A \epsig,\bs\xi)_\Th,\\
\nonumber
(\nabla\cdot\esig,\mathbf w)_\Th+\imath\kappa (\rho\eu,\mathbf w)_\Th\\
\label{eq:5.7b}
	+\langle\bs\tau(\PM \eu-\euhat),\PM\mathbf w\rangle_{\partial\Th} 
	&=\imath\kappa(\rho\epu,\mathbf w)_\Th\\
\nonumber
	&\phantom{=}+\langle \epsig\mathbf n,\mathbf w\rangle_{\partial\Th}
		+\langle\bs\tau\epu,\PM\mathbf w\rangle_{\partial\Th},\\
\label{eq:5.7c}
\langle \esig\mathbf n+\bs\tau(\PM\eu-\euhat),\bs\mu\rangle_{\partial\Th\setminus\Gamma_D}
	&=\langle\epsig\mathbf n,\bs\mu\rangle_{\partial\Th\setminus\Gamma_D}
		+\langle\bs\tau \PM\epu,\bs\mu\rangle_{\partial\Th\setminus\Gamma_D}\\
\label{eq:5.7d}
\langle\euhat,\bs\mu\rangle_{\Gamma_D} &=0,
\end{alignat}
\end{subequations}
for all $(\bs\xi,\mathbf w,\bs\mu)\in\Vh\times \Wh\times \Mh$. Clearly \eqref{eq:5.7d} is equivalent to $\euhat=\mathbf 0$ on $\Gamma_D$. We now (i) test \eqref{eq:5.7a} with $\bs\xi=\overline\esig$, (ii) conjugate \eqref{eq:5.7b} and test the result with $\mathbf w=\overline\eu$, (iii) conjugate \eqref{eq:5.7c} and test the result with $\bs\mu=-\overline\euhat$. The results of these three steps are added and reorganized (\eqref{eq:5.7d} is used for this) to prove the proposition.
\end{proof}

We next bound the last two terms in the right-hand side of the identity in Proposition \ref{prop:5.3}.
From this moment on, we will frequently use, without additional warning, approximation properties of the $L^2$ projections onto the space of piecewise polynomial functions.

\begin{proposition}\label{prop:5.4}
If $k\ge 1$, 
\begin{alignat*}{4}
 \left|\langle\overline\epsig\mathbf n,\eu-\euhat\rangle_{\partial\Th}
	+\langle\bs\tau \, \overline\epu, \PM \eu-\euhat\rangle_{\partial\Th}\right|
& \le C_1 
	\left( h^t |\bs\sigma |_{t,\Omega} + h^{s-1} |\mathbf u|_{s,\Omega}\right) \| \PM \eu-\euhat\|_\tau \\
&  \phantom{\le}
	+C_2 \kappa   h^t |\bs\sigma|_{t,\Omega}  
		\left(h^t |\bs\sigma|_{t,\Omega} +\|\esig\|_{\mathcal A}\right),
\end{alignat*}
if $\bs\sigma\in H^t(\Omega;\C^{3\times 3})$ for $1\le t \le k+1$ and $\mathbf u\in H^s(\Omega;\C^3)$ for $1\le s\le k+2$.
\end{proposition}

\begin{proof}
Following \cite[Lemma 4.3]{QiSh:2014}, it is possible to show that
\begin{subequations}\label{eq:5.8}
\begin{eqnarray}
	\label{eq:5.8a}
\left|\langle\overline\epsig\mathbf n,\eu-\euhat\rangle_{\partial\Th}\right|
	& \le & C h^t |\bs\sigma |_{t,\Omega} 
		\left( \| \PM \eu-\euhat\|_\tau + \| \bs\varepsilon(\eu)\|_\Th\right)  \\
\left| \langle\bs\tau \, \overline\epu, \PM \eu-\euhat\rangle_{\partial\Th}\right|
	&\le & C h^{s-1}|\mathbf u|_{s,\Omega}\| \PM \eu-\euhat\|_\tau.
	\label{eq:5.8b}
\end{eqnarray}
\end{subequations}
We recall that the argument leading to the proof of \eqref{eq:5.8b} needs the traces of rigid motions to be in $\Mh$, which is where the additional hypothesis $k\ge 1$ is used.

We next test the first error equation \eqref{eq:5.7a} with $\bs\xi=\bs\varepsilon(\overline\eu)$ restricted to $K$ to obtain
\begin{equation}\label{eq:5.9}
\| \bs\varepsilon(\eu)\|_K^2 = (\nabla \eu,\bs\varepsilon(\overline\eu))_K
	=\imath\kappa ( \mathcal A(\epsig-\esig),\bs\varepsilon(\overline\eu))_K
		+\langle \PM \eu-\euhat,\bs\varepsilon(\overline\eu)\mathbf n\rangle_{\partial K}.
\end{equation}
The scaling hypothesis on $\bs\tau$ given in \eqref{eq:tau} and a scaling argument using the fact that $\bs\varepsilon(\eu)$ is a polynomial on $K$ show then
\begin{eqnarray*}
\left|\langle \PM \eu-\euhat,\bs\varepsilon(\overline\eu)\mathbf n\rangle_{\partial K}\right|
	&\le & \| \PM \eu-\euhat\|_{\partial K} \| \bs\varepsilon(\eu)\|_{\partial K} \\
	&\le & C \| \bs\tau_K^{1/2} (\PM \eu-\euhat)\|_{\partial K} 
		h_K^{1/2} \| \bs\varepsilon(\eu)\|_{\partial K} \\
	&\le & C' \| \bs\tau_K^{1/2} (\PM \eu-\euhat)\|_{\partial K}  \| \bs\varepsilon(\eu)\|_K.
\end{eqnarray*}
Substituting these bounds in the right-hand side of \eqref{eq:5.9}, and adding over all elements, using \eqref{eq:5.3new} (the spectral bounds on the inverse compliance tensor) it follows that
\begin{equation}\label{eq:5.10}
\| \bs\varepsilon(\eu)\|_\Th \le C \kappa \left( \| \epsig\|_{\mathcal A}+\| \esig\|_{\mathcal A}\right)
			+ C \| \PM \eu-\euhat\|_\tau.
\end{equation}
Plugging \eqref{eq:5.10} in \eqref{eq:5.8}, the proposition is proved.
\end{proof}

\section{Dual problem and bootstrapping process}

We consider the adjoint system to \eqref{eq:2.2}:
\begin{subequations}\label{eq:6.1}
\begin{alignat}{6}
-\imath\kappa\mathcal A\bs\psi -\bs\varepsilon(\bs\phi) &=\mathbf 0
	&\qquad &\mbox{in $\Omega$},\\
-\nabla\cdot\bs\psi -\imath\kappa\,\rho\,\bs\phi &=\eu
	&\qquad &\mbox{in $\Omega$},\\
\bs\phi &=\mathbf 0
	&\qquad &\mbox{on $\Gamma_D$},\\
\bs\psi\mathbf n &=\mathbf 0
	&\qquad &\mbox{on $\Gamma_N$}.
\end{alignat}
\end{subequations}
This problem is uniquely solvable if \eqref{eq:3.3}, or equivalently, \eqref{eq:2.1}, is. If we assume regularity for the coercive problem \eqref{eq:3.1}, expressed in the bound \eqref{eq:3.2}, and assume the local smoothness of the coefficients given at the beginning of Section \ref{sec:3}, then it is easy to see that $\bs\phi\in H^2(\Omega;\mathbb C^3)$ and $\bs\psi\in H^1(\Omega;\mathbb C^{3\times3})$. Morever, we can bound 
\begin{equation}\label{eq:6.2}
\| \bs\phi\|_{1,\Omega} + \| \rho\bs\phi\|_{1,\Th}+\| \mathcal A \bs\psi\|_{1,\Omega}+ \|\bs\psi\|_{1,\Th}
+\kappa^{-1} \|\bs\phi\|_{2,\Omega} \le D_\kappa\|\eu\|_\Omega,
\end{equation}
where $D_\kappa:= C\,(1+\kappa C_\kappa+C_\kappa)$, $C_\kappa$ being the constant in \eqref{eq:3.4} and $C$ being allowed to depend on the regularity constant \eqref{eq:3.2} as well as on the physical coefficients. Here the norm $\|\cdot\|_{1,\Th}$ is the natural norm of the broken Sobolev space $\prod_{K\in \Th} H^1(K)$. 
We note that, while the regularity requirement can be somewhat relaxed, the analysis in this paper (see also \cite{QiSh:2014} and \cite{GrMo:2011}) needs a certain amount of regularity for the solution of the dual problem (which can be translated to regularity of the solution of \eqref{eq:3.1}) due to the need of having square integrable normal traces of $\bs\psi$ on the faces of the elements.

\begin{proposition}[Duality identity]
\label{prop:6.1}
\begin{eqnarray*}
\| \eu\|_\Omega^2
	&=& \imath\kappa \left( (\mathcal A\esig,\overline{\bs\psi}-\PV\overline{\bs\psi})_\Th+(\mathcal A\epsig,\PV\overline{\bs\psi})_\Th\right)\\	
	& & +\imath\kappa \left((\rho\,\eu,\overline{\bs\phi}-\PW\overline{\bs\phi})_\Th+(\rho\,\epu,\PW\overline{\bs\phi})_\Th\right)\\
	& & +\langle \eu-\euhat,(\overline{\bs\psi}-\PV\overline{\bs\psi})\mathbf n\rangle_{\partial\Th}\\
	& & +\langle\bs\tau(\PM \eu-\euhat)-\bs\tau\PM\epu,\overline{\bs\phi}-\PW\overline{\bs\phi}\rangle_{\partial\Th}
	       + \langle \epsig\mathbf n,\PW\overline{\bs\phi}-\PM\overline{\bs\phi}\rangle_{\partial\Th}.
\end{eqnarray*}
\end{proposition}

\begin{proof}
The proof is similar to duality arguments in \cite{QiSh:2014} (and related references). We give here a very systematic approach to help understand the logic of the argument. We first conjugate equations \eqref{eq:6.1} and then test them with the discrete errors $(\esig,\eu,\euhat)$:
\begin{subequations}\label{eq:6.2D}
\begin{alignat}{6}
\label{eq:6.2a}
\imath\kappa(\mathcal A\esig,\overline{\bs\psi})_\Th
	+(\nabla\cdot\esig,\PW\overline{\bs\phi})_\Th-\langle\esig\mathbf n,\overline{\bs\phi}\rangle_{\partial\Th} &=0,\\
\label{eq:6.2b}
-(\eu,\nabla\cdot\PV\overline{\bs\psi})_\Th+\langle \eu,(\PV\overline{\bs\psi}-\overline{\bs\psi})\mathbf n\rangle_{\partial\Th}
	+\imath\kappa (\rho\,\eu,\overline{\bs\phi})_\Th & = \| \eu\|_\Omega^2,\\
\label{eq:6.2c}	
\langle \euhat,\overline{\bs\psi}\mathbf n\rangle_{\partial\Th} &=0.
\end{alignat}
\end{subequations}
Note that to reach \eqref{eq:6.2a} and \eqref{eq:6.2b} we need to integrate by parts and introduce projections wherever possible. Also, \eqref{eq:6.2c} reflects the fact that $\bs\psi$ does not jump across interelement faces as well as the equality $\euhat=\mathbf 0$ on $\Gamma_D$. The second ingredient for the proof is the set of error equations \eqref{eq:5.7} tested with $(\PV\overline{\bs\psi},\PW\overline{\bs\phi},\PM\overline{\bs\phi})$, to yield
\begin{subequations}\label{eq:6.3}
\begin{alignat}{6}
\imath\kappa (\mathcal A \esig,\overline{\bs\psi})_\Th-(\eu,\nabla\cdot \PV\overline{\bs\psi})_\Th
	+\langle \euhat,\overline{\bs\psi}\mathbf n\rangle_{\partial\Th}
	&=\ell_1(\bs\psi),\\
(\nabla\cdot\esig,\PW\overline{\bs\phi})_\Th+\imath\kappa (\rho\eu,\overline{\bs\phi})_\Th
	&=\ell_2(\bs\phi),\\
-\langle \esig\mathbf n,\overline{\bs\phi}\rangle_{\partial\Th}
	&=\ell_3(\bs\phi),
\end{alignat}
\end{subequations}
where
\begin{eqnarray*}
\ell_1(\bs\psi)&:=&\imath\kappa \left(
	(\mathcal A\esig,\overline{\bs\psi}-\PV\overline{\bs\psi})_\Th
	+(\mathcal A \epsig,\PV\overline{\bs\psi})_\Th\right)
	+\langle\euhat,(\overline{\bs\psi}-\PV\overline{\bs\psi})\mathbf n\rangle_{\partial\Th},\\
\ell_2(\bs\phi) &:=& \imath\kappa\left(
	(\rho\eu,\overline{\bs\phi}-\PW\overline{\bs\phi})_\Th
	+(\rho\epu,\PW\overline{\bs\phi})_\Th\right)\\
	 & & +\langle \epsig\mathbf n,\PW\overline{\bs\phi}\rangle_{\partial\Th}
		+\langle\bs\tau\PM\epu,\PW\overline{\bs\phi}\rangle_{\partial\Th}
		 -\langle\bs\tau(\PM \eu-\euhat),\PW\overline{\bs\phi}\rangle_{\partial\Th} \\
\ell_3(\bs\phi) &:=& -\langle\epsig\mathbf n,\PM\overline{\bs\phi}\rangle_{\partial\Th}
		-\langle\bs\tau \PM\epu,\overline{\bs\phi}\rangle_{\partial\Th}
		+\langle \bs\tau(\PM\eu-\euhat),\overline{\bs\phi}\rangle_{\partial\Th}.
\end{eqnarray*}
Note that in \eqref{eq:6.3} we have kept in the left-hand side of the error equations only those terms that appear in the left-hand side of \eqref{eq:6.2D}. We have also eliminated some redundant projections and applied that $\bs\phi=\mathbf 0$ on $\Gamma_D$. The proof of the result is now straightforward: add equations \eqref{eq:6.2D} and substitute equations \eqref{eq:6.3} in the result.
\end{proof}

The next step in the proof of the error estimates is a bound for $\|\eu\|_\Omega$ obtained by carefully working on the right-hand side of the duality identity in Proposition \ref{prop:6.1}. To alleviate the proof from an excess of constants, we will use the convention that $a\lesssim b$, whenever there exists a positive constant $C$ independent of $h$ and $\kappa$ such that $a\le C \, b$. 

\begin{proposition}\label{prop:6.2}
If $h \kappa D_\kappa$ is small enough if $\bs\sigma\in H^t(\Omega;\C^{3\times 3})$ for $1\le t \le k+1$ and if $\mathbf u\in H^s(\Omega;\C^3)$ for $1\le s\le k+2$, then
\[
\| \eu\|_\Omega \lesssim  h (\kappa+1) D_\kappa 
\left( \| \esig\|_{\mathcal A}+\| \PM \eu-\euhat\|_\tau
	+ h^t |\bs\sigma|_{t,\Omega}
	+h^{s-1} |\mathbf u|_{s,\Omega}\right).
\]
\end{proposition}

\begin{proof}
As already mentioned, we first estimate the right-hand side of the equality in Proposition \ref{prop:6.1}. Let $\av{f}$ be the best $L^2(\Omega)$ projection of $f$ on the space of piecewise constant functions, i.e., $\av{f}=\frac1{|K|}\int_K f$ in $K$ for every $K$. 

Notice that
\[
(\mathcal A\esig,\overline{\bs\psi}-\PV\overline{\bs\psi})_\Th
	+(\mathcal A\epsig,\PV\overline{\bs\psi})_\Th 
	= (\mathcal A(\bs\sigma-\bs\sigma_h),\overline{\bs\psi}-\PV\overline{\bs\psi})_\Th
		+ (\epsig,\mathcal A\overline{\bs\psi}-\av{\mathcal A\overline{\bs\psi}})_\Th,
\]
and then we can bound
\begin{equation}\label{eq:6.5}
|(\mathcal A\esig,\overline{\bs\psi}-\PV\overline{\bs\psi})_\Th 
	+(\mathcal A\epsig,\PV\overline{\bs\psi})_\Th |
	\lesssim h \|\bs\sigma-\bs\sigma_h\|_{\mathcal A}  |\bs\psi|_{1,\Th}
		+ h \|\epsig\|_\Omega  |\mathcal A \bs\psi|_{1,\Omega}.
\end{equation}
Similarly, the equality
\[
(\rho\,\eu,\overline{\bs\phi}-\PW\overline{\bs\phi})_\Th+(\rho\,\epu,\PW\overline{\bs\phi})_\Th
=(\rho(\mathbf u-\mathbf u_h),\overline{\bs\phi}-\PW\overline{\bs\phi})_\Th
	+(\epu,\rho\overline{\bs\phi}-\av{\rho\overline{\bs\phi}})_\Th
\]
can be used to estimate
\begin{equation}\label{eq:6.6}
|(\rho\,\eu,\overline{\bs\phi}-\PW\overline{\bs\phi})_\Th+(\rho\,\epu,\PW\overline{\bs\phi})_\Th|
\lesssim h \|\mathbf u-\mathbf u_h\|_\Omega | \bs\phi|_{1,\Omega}
	+ h \|\epu\|_\Omega |\rho\,\bs\phi|_{1,\Th}
\end{equation}
Using \eqref{eq:5.8a} with $\bs\psi$ in place of $\bs\sigma$ and $t=1$ and \eqref{eq:5.10}, we can estimate
\begin{equation}\label{eq:6.7}
|\langle \eu-\euhat,(\overline{\bs\psi}-\PV\overline{\bs\psi})\mathbf n\rangle_{\partial\Th}|
\lesssim h \left( \| \epsig\|_{\mathcal A}+\| \esig\|_{\mathcal A}+ \| \PM \eu-\euhat\|_\tau\right)
|\bs\psi|_{1,\Th}.
\end{equation}
Using \eqref{eq:5.8b} with $\bs\phi$ in place of $\mathbf u$ and $s=2$, we bound
\begin{equation}\label{eq:6.8}
|\langle\bs\tau(\PM \eu-\euhat),\overline{\bs\phi}-\PW\overline{\bs\phi}\rangle_{\partial\Th}|
\lesssim h \|\PM \eu-\euhat\|_\tau | \bs\phi|_{2,\Omega}.
\end{equation}
With a scaling argument and the bound \eqref{eq:tau} (stating the size of the stabilization paraneter), we can bound on every $K$
\begin{eqnarray*}
|\langle \bs\tau \PM (\mathbf u-\PW \mathbf u),\overline{\bs\phi}-\PW\overline{\bs\phi}\rangle_{\partial K}|
	&\lesssim & h_K^{-1} \| \mathbf u-\PW \mathbf u\|_{\partial K}
						\| \bs\phi-\PW\bs\phi\|_{\partial K}\\
	& \lesssim & h_K^s |\mathbf u|_{s,K}|\bs\phi|_{2,K}
\end{eqnarray*}
and therefore
\begin{equation}\label{eq:6.9}
|\langle \bs\tau\PM\epu,\overline{\bs\phi}-\PW\overline{\bs\phi}\rangle_{\partial\Th}|
\lesssim h^s |\mathbf u|_{s,\Omega} |\bs\phi|_{2,\Omega}.
\end{equation}
Similarly
\begin{equation}\label{eq:6.10}
|\langle \epsig\mathbf n,\PM\overline{\bs\phi}-\PW\overline{\bs\phi}\rangle_{\partial\Th}|
\lesssim h^{t+1} |\bs\sigma|_{t,\Omega} |\bs\phi|_{2,\Omega}.
\end{equation}
Collecting the estimates \eqref{eq:6.5}-\eqref{eq:6.10} to bound the right-hand side of the identity in Proposition \ref{prop:6.1}, and using the regularity bound \eqref{eq:6.2}, we can bound
\begin{align*}
\| \eu\|_\Omega^2
	\lesssim  \, & h \kappa D_\kappa \| \eu\|_\Omega 
	\left( \| \eu\|_\Omega + \|\esig\|_{\mathcal A} +\| \PM \eu -\euhat\|_\tau \right) \\
	& + h (1+\kappa ) D_\kappa\left(h^t |\bs\sigma|_{t,\Omega}  + h^{s-1} |\mathbf u|_{s,\Omega} \right)
\end{align*}
The proposition is now a simple consequence of the latter inequality.
\end{proof}

The proof of Theorem \ref{the:3.1} follows from the energy identity (Proposition \ref{prop:5.3}) and the estimates of Propositions \ref{prop:5.4} and \ref{prop:6.2} by a careful bootstrapping process.

\begin{proof}[Proof of Theorem \ref{the:3.1}]
To simplify the algebra involved in this final step, let us give symbols for the quantities we want to bound
\[
\Sigma:=\| \esig\|_{\mathcal A}, \qquad 
	\mathrm T:=\kappa^{-1/2}\| \PM \eu-\euhat\|_\tau, \qquad
	\mathrm U:=\| \eu\|_\Omega,
\]
and for the approximation terms
\[
\Sigma_h:= h^t |\bs\sigma|_{t,\Omega},\qquad
	\mathrm U_h:=h^{s-1} |\mathbf u|_{s,\Omega}.
\] 
With this shorthand, Propositions \ref{prop:5.3} and \ref{prop:5.4} yield
\begin{equation}\label{eq:6.11}
|\Sigma^2-\imath \mathrm T^2|
	\lesssim \Sigma\,\Sigma_h + \mathrm U\,\mathrm U_h 
	+ \kappa^{-1/2} (\Sigma_h+\mathrm U_h) \mathrm T +\Sigma_h^2+\mathrm U^2.
\end{equation}
If $\alpha:=h (1+\kappa) (1+\kappa C_\kappa)=D_\kappa h (1+\kappa)$, Proposition \ref{prop:6.2} can then be rephrased as
\begin{equation}\label{eq:6.12}
\mathrm U \lesssim \alpha (\Sigma+\Sigma_h+\mathrm U_h +\kappa^{1/2}\mathrm T).
\end{equation}
Substituting \eqref{eq:6.12} in the right-hand side of \eqref{eq:6.11}, and reordering terms, we have
\begin{eqnarray}\label{eq:6.13}
\nonumber
\Sigma^2+\mathrm T^2 
	& \lesssim &
	\alpha^2\Sigma^2+\Sigma(\Sigma_h+\alpha \mathrm U_h)\\
	&& +\alpha^2 \kappa \mathrm T^2
	+\mathrm T(\kappa^{-1/2}\mathrm U_h+\alpha \kappa^{1/2}\mathrm U_h
				+\kappa^{-1/2} \Sigma_h)\\
\nonumber
	& &
	+(1+ \alpha^2) \Sigma_h^2
	 + (\alpha+\alpha^2)\mathrm U_h^2.
\end{eqnarray}
Let now $C$ be the constant that is hidden in the symbol $\lesssim$, and let us assume that
\begin{equation}\label{eq:hypo}
C \alpha^2\le \tfrac14 \qquad \mbox{and} \qquad C \alpha^2\kappa\le \tfrac14.
\end{equation}
We now use Young's inequality $ab\le \frac14 a^2+ b^2$ in \eqref{eq:6.13} to get
\begin{eqnarray*}
\Sigma^2+\mathrm T^2  
		& \le & \tfrac12\Sigma^2+\tfrac12 \mathrm T^2 \\
		& & + C^2(\Sigma_h+\alpha \mathrm U_h)^2
			+C^2(\kappa^{-1/2}\mathrm U_h+\alpha \kappa^{1/2}\mathrm U_h+\kappa^{-1/2} \Sigma_h)^2\\
		& & +C(1+ \alpha^2) \Sigma_h^2
	 		+ C(\alpha+\alpha^2)\mathrm U_h^2.
\end{eqnarray*}
We can now simplify this expression using \eqref{eq:hypo} to obtain
$\Sigma^2+\mathrm T^2 \lesssim (1+\kappa^{-1})(\Sigma_h^2+\mathrm U_h^2)$, 
or equivalently
\begin{equation}\label{eq:6.14}
\Sigma+\mathrm T 
	\lesssim 
	(1+\kappa^{-1/2})(\Sigma_h+\mathrm U_h).		
\end{equation}
Using \eqref{eq:6.14} in \eqref{eq:6.12}, we can finally prove that
\begin{equation}\label{eq:6.15}
\mathrm U \lesssim \alpha(1+\kappa^{1/2}+\kappa^{-1/2})(\Sigma_h+\mathrm U_h).
\end{equation}
This finishes the proof.
\end{proof}

\section{Variants and insights}\label{sec:6}

\paragraph{Matrix form.}
We first give a matrix representation of the method of Section \ref{sec:3}. Equations \eqref{eq:2.5c} and \eqref{eq:2.5d} suggest the following orthogonal decomposition $\Mh=\Mhh{nD}\oplus \Mhh{D}$, where $\Mhh{nD}=\{\bs\mu\,:\,\bs\mu|_{\Gamma_D}=0\}\equiv \{ \bs\mu|_{\partial\Th\setminus\Gamma_D}\,:\,\bs\mu\in \Mh\}$. 	
We now take real-valued bases for the spaces $\Vh$, $\Wh$, $\Mhh{nD}$, and $\Mhh{D}$ and identify the unknowns $\bs\sigma_h\in\Vh$, $\bff u_h\in\Wh$, $\widehat{\bff u}_h|_{\partial\Th\setminus\Gamma_D}\in \Mhh{nD}$, and $\widehat{\bff u}_h|_{\Gamma_D}\in \Mhh{D}$, with respective complex column vectors $\underline\sigma$, $\underline u$, $\underline{\widehat u}^{{nD}}$, and $\underline{\widehat u}^{{D}}$. We then consider {\it real matrices} associated to the following {\it bilinear forms} that are understood as functionals acting on the unknowns:
\begin{alignat*}{6}
(\mathcal A \bs\sigma_h,\bs\xi)_\Th& \qquad  & \bs\xi\in \Vh & \qquad & & \mathrm A \underline\sigma, \\
(\nabla\cdot\bs\sigma_h,\bff w)_\Th & & \bff w\in \Wh & & & \mathrm D\underline\sigma, \\
\langle\bs\sigma_h\bff n,\bs\mu\rangle_{\partial\Th\setminus\Gamma_D} & & \bs\mu\in \Mhh{nD} 
		& & &\mathrm N \underline\sigma,\\
\langle\bs\sigma_h\bff n,\bs\mu\rangle_{\Gamma_D} & & \bs\mu\in \Mhh{D} 
		& & &\mathrm N_{{D}} \underline\sigma,\\
(\rho\bff u_h,\bff w)_\Th & & \bff w\in \Wh & & & \mathrm M\underline u,\\
\langle \bs\tau\PM \bff u_h,\PM\bff w\rangle_{\Th} & & \bff w\in \Wh & & & \mathrm T_{11} \underline u,\\
\langle\bs\tau\widehat{\bff u}_h,\bff w_h\rangle_{\partial\Th\setminus\Gamma_D} & & \bff w\in \Wh 
		& & & \mathrm T_{12}\underline u^{{nD}},\\
\langle\bs\tau\widehat{\bff u}_h,\bff w_h\rangle_{\Gamma_D} & & \bff w\in \Wh 
		& & & \mathrm T_{D}\underline u^{{D}},\\
\langle\bs\tau\widehat{\bff u}_h,\bs\mu\rangle_{\partial\Th\setminus\Gamma_D} & & \bs\mu\in \Mhh{nD} 
		& & & \mathrm T_{22}\underline u^{{nD}},\\
\langle\widehat{\bff u}_h,\bs\mu\rangle_{\Gamma_D} & &\bs\mu\in \Mhh{D}
		& & & \mathrm M_{{D}} \underline{\widehat u}^{{D}}.
\end{alignat*}Note that the matrices $\mathrm A$, $\mathrm M$, and $\mathrm T_{22}$ are symmetric and positive definite, while $\mathrm T_{11}$ is symmetric and positive semidefinite.  
The method given by equations \eqref{eq:2.5} is then equivalent to the linear system
\begin{equation}\label{eq:7.0}
\left[\begin{array}{cccc}
	\imath\kappa \mathrm A & -\mathrm D^\top & \mathrm N^\top & \mathrm N_{{D}}^\top \\
	\mathrm D & \imath\kappa\mathrm M+\mathrm T_{11} & -\mathrm T_{12} & -\mathrm T_{D} \\
	-\mathrm N & -\mathrm T_{12}^\top & \mathrm T_{22} & \mathrm O \\
	\mathrm O & \mathrm O &\mathrm O & \mathrm M_{{D}}
\end{array}\right]
\left[\begin{array}{l} 
	\underline\sigma \\ \underline u\\ \underline{\widehat u}^{{nD}} \\ \underline{\widehat u}^{{D}}
\end{array}\right]
=
\left[\begin{array}{c}
	\underline 0 \\ \underline f \\ -\underline g_N \\ \underline g_D
\end{array}\right],
\end{equation}
where the definition of the vectors $\underline f$, $\underline g_N$ and $\underline g_D$ is self-evident. 
Equations \eqref{eq:7.0} are equivalent to the following system
\[
\left[\begin{array}{cccc}
	\mathrm A & \mathrm D^\top & -\mathrm N^\top & -\mathrm N_{{D}}^\top \\
	-\mathrm D & -\kappa^2\mathrm M-\imath\kappa\mathrm T_{11} & 
			\imath\kappa\mathrm T_{12} & \imath\kappa\mathrm T_{D} \\
	\mathrm N & -\imath\kappa\mathrm T_{12}^\top & \imath\kappa\mathrm T_{22} & \mathrm O \\
	\mathrm O & \mathrm O &\mathrm O & \mathrm M_{{D}}
\end{array}\right]
\left[\begin{array}{l} 
	\widetilde{\underline\sigma} \\ \underline u\\ \underline{\widehat u}^{{nD}} \\ \underline{\widehat u}^{{D}}
\end{array}\right]
=
\left[\begin{array}{c}
	\underline 0 \\ \widetilde{\underline f} \\ \widetilde{\underline g}_N \\ \underline g_D
\end{array}\right],
\]
where $\widetilde{\underline\sigma}=-\imath\kappa\underline\sigma$, $\widetilde{\underline f}=-\imath\kappa \underline f$, and $\widetilde{\underline g}_N=-\imath\kappa\underline g_N$. This change of variables in the first unknown and in the right-hand side reverts the system to the original physical variables (the ones with a tilde in Section \ref{sec:2}) so that the equations are second-order-in-frequency. It is clear how the stabilization terms are the only complex-valued ones in the system.

\paragraph{Hybridization.} The four matrices in the upper left $2\times 2$ block of the matrix of \eqref{eq:7.0} are elementwise block diagonal. The hybridization process consists of solving the system
\begin{alignat*}{6}
\left( \left[\begin{array}{cc} \mathrm N & \mathrm T_{12}^\top \end{array}\right] \mathrm C^{-1} 
\left[\begin{array}{c} \mathrm N^\top \\ - \mathrm T_{12} \end{array}\right]+\mathrm T_{22} \right) \underline{\widehat u}^{{nD}}
 = & -\underline g_N + 
 \left[\begin{array}{cc} \mathrm N & \mathrm T_{12}^\top \end{array}\right] \mathrm C^{-1} 
\left[\begin{array}{c} \underline 0 \\ \underline f \end{array}\right] \\
 & - \left[\begin{array}{cc} \mathrm N & \mathrm T_{12}^\top \end{array}\right] \mathrm C^{-1} 
\left[\begin{array}{c} \mathrm N_D^\top \\ - \mathrm T_D \end{array}\right]\mathrm M_D^{-1} \underline g_D,
\end{alignat*}
where the invertibility of
\[
\mathrm C:= 
	\left[\begin{array}{cc} 
		\imath \kappa \mathrm A & -\mathrm D^\top \\
		\mathrm D & \imath \kappa\mathrm M+\mathrm T_{11}
	\end{array}\right]			
\]
was the object of Proposition \ref{prop:5.2}.

\paragraph{Variant \# 1: time reversal.}
While we are not making any claims about the behavior of the method for high frequency problems, we have kept $\kappa$ visible everywhere. We next explore some variants of the method that can be obtained by changing to second-order-in-frequency form and exploring different choices of the stabilization parameter.
The energy identity of Proposition \ref{prop:5.3} is the trigger for the analysis of the method. It is there clear that the sign of the boundary term is not relevant and a method based on the numerical flux
\[
\widehat{\bs\sigma}_h\mathbf n:=\bs\sigma_h\mathbf n-\bs\tau_K (\PM \mathbf u_h-\widehat{\mathbf u}_h) 
\,:\, \partial K\to \C^3
\]
has the same convergence properties as the method presented in Section \ref{sec:2}. As we will see later in this section, this method corresponds to time reversal.

\paragraph{Variant \# 2: $\kappa$-scaled stabilization.}
The factor $\kappa^{-1/2}$ in the error estimate of Theorem \ref{the:3.1} suggests the following variant of the numerical method: we still use equations \eqref{eq:2.5} by change the definition of the numerical flux to be
\begin{equation}\label{eq:7.1}
\widehat{\bs\sigma}_h\mathbf n:=\bs\sigma_h\mathbf n+\kappa \bs\tau_K (\PM \mathbf u_h-\widehat{\mathbf u}_h) 
\,:\, \partial K\to \C^3.
\end{equation}
(Note that, as shown in \cite{GoLaNiPe:2015} for the acoustic wave equation, making the stabilization parameter depend on $\kappa$ is a must when we want to deal with complex frequencies. This dependence has also some desirable properties.) 
The proof of Theorem \ref{the:3.1} can be easily adapted to deal with the method whose stabilization term is given by \eqref{eq:7.1}. The error estimate is given in the following theorem.

\begin{theorem}
There exist $C_1, C_2>0$, dependent only on the shape-regularity of $\Th$, the density $\rho$ and the coefficients of the inverse compliance tensor $\mathcal A$ such that
if $h (1+\kappa)^{3/2}(1+\kappa C_\kappa)$ is small enough, then the errors can be bounded by
\[
\| \esig\|_{\mathcal A}+ \| \PM \eu-\euhat\|_\tau
	\le C_1 \big((1+\kappa^{-1})  h^t |\bs\sigma|_{t,\Omega}+h^{s-1} |\mathbf u|_{s,\Omega}\big)
\]
and
\[
\| \eu\|_\Omega\le C_2 (1+\kappa C_\kappa) (1+\kappa) 
\big((\kappa+\kappa^{-1}) h^{t+1} |\bs\sigma|_{t,\Omega}+h^{s} (1+\kappa) |\mathbf u|_{s,\Omega}\big),
\]
if $k\ge 1$, $\mathbf u\in H^s(\Omega;\C^3)$ with $1\le s\le k+2$, and $\bs\sigma\in H^t(\Omega;\C^{3\times 3})$ with $1\le t\le k+1$.
\end{theorem}

\paragraph{Second-order-in-frequency formulations.}
Since all methods presented above are based in a first-order-in-frequence formulation, defining $\bs\sigma:=(\imath/\kappa)\widetilde{\bs\sigma}$, where $\widetilde{\bs\sigma}$ is the physical stress for the displacement field $\bff u$. Consider now the following family of HDG schemes based on a second-order-in-frequency formulation: the spaces are unchanged and $\alpha_\kappa$ is a fixed parameter that is allowed to depend on the frequency:
\begin{subequations}\label{eq:7.2}
\begin{alignat}{6}
 (\mathcal A\widetilde{\bs\sigma}_h,\bs\xi)_{\Th}
	+(\mathbf u_h,\nabla\cdot\bs\xi)_\Th-\langle\widehat{\mathbf u}_h,\bs\xi\mathbf n\rangle_{\partial\Th}
	 & = 0 & \qquad & \forall \bs\xi\in \Vh,\\
(\widetilde{\bs\sigma}_h,\nabla\mathbf w)_{\Th}
	- \langle \widehat{\bs\sigma}_h\bff n,\mathbf w\rangle_{\bs\tau}
	-\kappa^2 (\rho\,\mathbf u_h,\mathbf w)_\Th &=-(\widetilde{\mathbf f},\mathbf w)_\Th
		& & \forall \mathbf w\in \Wh,\\
\langle\widehat{\bs\sigma}_h\bff n,\bs\mu\rangle_{\partial\Th\setminus\Gamma_D}
	&=\langle\widetilde{\mathbf g}_N,\bs\mu\rangle_{\Gamma_N} & & \forall \bs\mu \in \Mh,\\
\langle\widehat{\mathbf u}_h,\bs\mu\rangle_{\Gamma_D} &=\langle\mathbf g_D,\bs\mu\rangle_{\Gamma_D}
	& &\forall \bs\mu\in \Mh,
\end{alignat}
where
\begin{equation}
\widehat{\bs\sigma}_h\bff n:=\widetilde{\bs\sigma}_h\bff n-\alpha_\kappa \bs\tau(\PM\bff u_h-\widehat{\bff u}_h).
\end{equation}
\end{subequations}
Note that the equations are written in terms of the original data in \eqref{eq:2.1}.
The choice $\alpha_\kappa=1$ is the direct application of the method in \cite{QiSh:2014} to the equation $\nabla\cdot\widetilde{\bs\sigma}+\kappa^2\rho\bff u=\widetilde{\bff f}$. This choice of the parameter $\alpha_\kappa$ yields a method that transitions smoothly (analytically) to the zero-frequency limit.
Methods based on the first-order-in-frequency formulation can be rewritten in the form \eqref{eq:7.2} with the relation $\widetilde{\bs\sigma}_h=-\imath\kappa\bs\sigma_h$ and the parameter $\alpha_\kappa:=\imath\kappa$ (for the method of Section \ref{sec:3}), $\alpha_\kappa:=-\imath\kappa$ (time reversed method) or $\alpha_\kappa=\imath\kappa^2$ (for the method with the flux defined in \eqref{eq:7.1}. 

\paragraph{Variant \# 3: conservative method.}
The error estimate for the method in \eqref{eq:7.2} with the choice $\alpha_\kappa=1$ is given in the next theorem. Not surprisingly the estimates hold as $\kappa\to 0$, since we end up with a smooth perturbation of the discretization of the steady-state equations. Note that when $\kappa\to 0$ and we are not dealing with the pure Neumann problem, the quantity $C_\kappa$ converges to a finite value. Later in this section we will see that this choice corresponds to a conservative method in the time domain.

\begin{theorem}\label{the:7.2}
There exist $C_1, C_2>0$, dependent only on the shape-regularity of $\Th$, the density $\rho$ and the coefficients of the inverse compliance tensor $\mathcal A$ such that
if $h (1+\kappa) E_\kappa$ is small enough, then the errors can be bounded by
\[
\| \PV\widetilde{\bs\sigma}-\widetilde{\bs\sigma}_h \|_{\mathcal A}+ \| \PM \eu-\euhat\|_\tau
	\le C_1 \big(  h^t |\widetilde{\bs\sigma}|_{t,\Omega}+(1+\kappa) h^{s-1} |\mathbf u|_{s,\Omega}\big)
\]
and
\[
\| \eu\|_\Omega\le C_2 E_\kappa (1+\kappa)
\big( h^{t+1} |\widetilde{\bs\sigma}|_{t,\Omega}+h^{s} (1+\kappa) |\mathbf u|_{s,\Omega}\big),
\]
if $k\ge 1$, $\mathbf u\in H^s(\Omega;\C^3)$ with $1\le s\le k+2$, and $\widetilde{\bs\sigma}\in H^t(\Omega;\C^{3\times 3})$ with $1\le t\le k+1$. Here $E_\kappa\le C (1+\kappa C_\kappa+ C_\kappa+\kappa C_\kappa^{1/2})$. 
\end{theorem}

\begin{proof} See Section \ref{sec:A}. \end{proof}

\paragraph{Methods in the time domain.} Let us now write some HDG-semidiscrete methods for the transient elastic wave equation. The data functions are $\widetilde{\bff f}:[0,\infty)\to L^2(\Omega;\R^3)$, $\bff g_D:[0,\infty)\to H^{1/2}(\Gamma_D;\R^3)$, and $\widetilde{\bff g}_N:[0,\infty)\to L^2(\Gamma_N;\R^3)$. We then look for $\bff u:[0,\infty)\to H^1(\Omega;\R^3)$ and $\bs\sigma:[0,\infty)\to H(\mathrm{div},\Omega;\R^{3\times 3}_{\mathrm{sym}})$ satisfying
\begin{subequations}\label{eq:6.4}
\begin{alignat}{6}
\mathcal A \widetilde{\bs\sigma}(t)-\bs\varepsilon (\mathbf u)(t)&=\bs 0&\qquad &\mbox{in $\Omega$}, \quad\forall t\ge 0,\\
\nabla\cdot\widetilde{\bs\sigma} -\rho\,\ddot{\mathbf u}(t)
	 &= \widetilde{\mathbf f}(t) &\qquad &\mbox{in $\Omega$}, \quad\forall t\ge 0,\\
\mathbf u(t) &=\mathbf g_D(t) & & \mbox{on $\Gamma_D$}, \quad\forall t\ge 0,\\
\widetilde{\bs\sigma}(t)\,\mathbf n &=\widetilde{\mathbf g}_N(t) & & \mbox{on $\Gamma_N$}, \quad\forall t\ge 0,
\end{alignat}
\end{subequations}
and initial conditions $\bff u(0)=\bff u_0$, $\dot{\bff u}(0)=\bff v_0$. The HDG semidiscretization consists of the search for $\widetilde{\bs\sigma}_h:[0,\infty)\to \Vh$, $\bff u_h:[0,\infty)\to \Wh$, and $\widehat{\bff u}_h:[0,\infty)\to \Mh$ satisfying
\begin{subequations}\label{eq:7.5}
\begin{alignat}{6}
\label{eq:7.5a}
 (\mathcal A \widetilde{\bs\sigma}_h(t),\bs\xi)_{\Th}
	+(\mathbf u_h(t),\nabla\cdot\bs\xi)_\Th-\langle\widehat{\mathbf u}_h(t),\bs\xi\mathbf n\rangle_{\partial\Th}
	 & = 0 & \\
\label{eq:7.5b}
(\widetilde{\bs\sigma}_h(t),\nabla\mathbf w)_{\Th}
	- \langle \widehat{\bs\sigma}_h(t)\bff n,\mathbf w\rangle_{\partial\Th}
	+ (\rho\,\ddot{\mathbf u}_h(t),\mathbf w)_\Th &=-(\widetilde{\mathbf f}(t),\mathbf w)_\Th
		& & \\
\label{eq:7.5c}
\langle\widehat{\bs\sigma}_h(t)\bff n,\bs\mu\rangle_{\partial\Th\setminus\Gamma_D}
	&=\langle\widetilde{\mathbf g}_N(t),\bs\mu\rangle_{\Gamma_N} & & \\
\label{eq:7.5d}
\langle\widehat{\mathbf u}_h(t),\bs\mu\rangle_{\Gamma_D} &=\langle\mathbf g_D(t),\bs\mu\rangle_{\Gamma_D}
	& &
\end{alignat}
\end{subequations}
for all $(\bs\xi,\bff w,\bs\mu)\in \Vh\times \Wh\times \Mh$ amd $t\ge 0$. The numerical flux $\widehat{\bs\sigma}_h$ can be defined in different ways which influence the choice of initial conditions. If we take the inverse Fourier transform of equations \eqref{eq:7.2} with $\alpha_\kappa=\pm \imath\kappa$, we obtain the following proposals for the numerical flux
\begin{equation}\label{eq:7.6}
\widehat{\bs\sigma}_h(t)\bff n:=\widetilde{\bs\sigma}_h(t)\bff n\pm \bs\tau(\PM\dot{\bff u}_h(t)-\dot{\widehat{\bff u}}_h(t)). 
\end{equation}
Note the positive sign corresponds to the method of Section \ref{sec:3} while the negative sign is the one obtained by time reversal. (It is clear from this why the sign change in the parameter $\alpha_\kappa$ corresponded to time reversal.) For equations \eqref{eq:7.2} with $\alpha_\kappa=1$ we obtain the flux
\begin{equation}\label{eq:7.7}
\widehat{\bs\sigma}_h(t)\bff n:=\widetilde{\bs\sigma}_h(t)\bff n+\bs\tau(\PM\bff u_h(t)-\widehat{\bff u}_h(t)). 
\end{equation}
The following result shows that: the method with flux given by \eqref{eq:7.6} with positive sign accumulates energy over time, the method with flux \eqref{eq:7.6} with negative sign is dissipative, and the method with flux \eqref{eq:7.7} is conservative. The build-up or dissipation of energy happens at the interfaces, while the conservative method needs to add a potential energy term in the interfaces. 

\begin{proposition}
Assume that the problem is unforced ($\widetilde{\mathbf f}=\mathbf 0$ and $\widetilde{\mathbf g}|_N=\mathbf 0$) and the Dirichlet boundary conditions are static $(\dot{\mathbf g}_D=\mathbf 0$). Then the solution to the HDG-semidiscrete equations \eqref{eq:7.5} with flux defined by \eqref{eq:7.6} satisfy
\[
\frac{\mathrm d}{\mathrm dt}
\left(\frac12 \| \widetilde{\bs\sigma}_h(t)\|_{\mathcal A}^2+\frac12 \| \dot{\bff u}_h(t)\|_\rho^2 \right)
=\pm \| \PM \dot{\bff u}_h(t)-\dot{\widehat{\bff u}}_h(t)\|_\tau^2 \qquad \forall t.
\]
The solution to the HDG equations \eqref{eq:7.5} with flux defined by \eqref{eq:7.7} satisfy
\[
\frac{\mathrm d}{\mathrm dt}
\left(\frac12 \| \widetilde{\bs\sigma}_h(t)\|_{\mathcal A}^2+\frac12 \| \dot{\bff u}_h(t)\|_\rho^2 
+\frac12 \|\PM \bff u_h(t)-\widehat{\bff u}_h(t)\|_\tau^2\right)=0 \qquad \forall t.
\]
\end{proposition}

\begin{proof} It follows from a simple argument: (a) differentiate \eqref{eq:7.5a} and test with $\widetilde{\bs\sigma}_h(t)$, (b) test \eqref{eq:7.5b} with $\dot{\bff u}_h(t)$, (c) test \eqref{eq:7.5c} with $\dot{\widehat{\bff u}}_h(t)$; finally add the result of (a)-(c) using the fact that $\widehat{\bff u}_h(t)$ is constant in time on the Dirichlet faces.
\end{proof}

\newpage
\section{Numerical experiments}

In order to have a good collection of meshes of the domain, we will conduct all the experiments in the unit cube $\Omega=(0,1)^3$. The cube will be partitioned in a sequence of tetrahedrizations obtained as follows: we first divide the cube along the coordinate planes into $n^3$ equally sized cubes (for $n=1,\ldots,7$), and then partition the cubes into six tetrahedra. We will tag the triangulation as $\Th$, where we can consider that $h=1/n$. The partitions corresponding to $n_1$ and $n_2$, where $n_2$ is a multiple of $n_1$, are nested. 

Tests will be run for constant and variable coefficients. In all cases we will consider isotropic materials, with $\mathcal C \bs\varepsilon=2\mu\bs\varepsilon+\lambda(\trace\bs\varepsilon) I_{3\times3}$, for constant or variable Lam\'e parameters $\lambda$ and $\mu$. In the case of constant coefficients we will take $\rho=\lambda=\mu=1$. In the variable case we take
\begin{alignat}{6}\label{eq:coeffsvar}
& \rho(\mathbf x) = 1+|\mathbf x|^2,\\
& \lambda(\mathbf x) = 2 + 0.2x_1^2 + 0.3x_2^2 + 0.04x_3^2,\\
& \mu(\mathbf x)= 3 + 0.5x_2^2 + 0.03x_3^2.
\end{alignat}

The code expands the three dimensional HDG Matlab package of \cite{FuGaSa:2013}. The code is written for the second-order-in-frequency form given in \eqref{eq:7.2}, which allows for easy comparisons with the physical quantities of interest and for almost straightforward changes to have all the methods we have considered in this paper. All the element-by-element operations (needed for building the local solvers and recovering displacement and stress after the hybridized system has been solved) are carried out in parallel. 

\paragraph{A problem with variable coefficients.}\label{sec:7.2}
We imbue the domain with the variable density and material coefficients given by \eqref{eq:coeffsvar}. Dirichlet boundary conditions are imposed on the top and bottom faces of the cube, while Neumann conditions are given on the remaining faces of $\Omega$. Data are built so that
\begin{equation*}
	\mathbf u(\mathbf x) = 
	\begin{bmatrix}
		\cos(\pi x_1) \sin (\pi x_2) \cos (\pi x_3) \\
		5x_1^2x_2x_3+4x_1x_2x_3+3x_1x_2x_3^2 + 17 		\\
		\cos(2x_2)\cos(3x_2)\cos(x_3)
	\end{bmatrix}
\end{equation*}
is the exact solution. Figure~\ref{fig:test1-fo} shows approximation errors for the method that is equivalent to the one applied to first-order-in-frequency
formulation, i.e., we set $\alpha_\kappa := \imath \kappa, \tau = h^{-1}$ in~\eqref{eq:7.2}.
We measure errors for the displacement and the `physical' stress
\begin{equation}\label{eq:errors}
\|\mathbf u-\mathbf u_h\|_\Omega, 
	\qquad
\|\widetilde{\bs\sigma}-\widetilde{\bs\sigma}_h\|_\Omega.
\end{equation}
Error plots in Figure \ref{fig:test1-fo} are shown for the method for different values of the polynomial degree (recall that $\bs\sigma$ is approximated with piecewise $\mathcal P_k$ functions, while $\mathbf u$ is approximated with piecewise $\mathcal P_{k+1}$ functions). The lowest order method ($k=0$), for which we do not have theoretical support shows a quite bad performance in practice and the corresponding results are not shown.  

\begin{figure}
\centering
\subfloat{\includegraphics[width=0.5\textwidth]{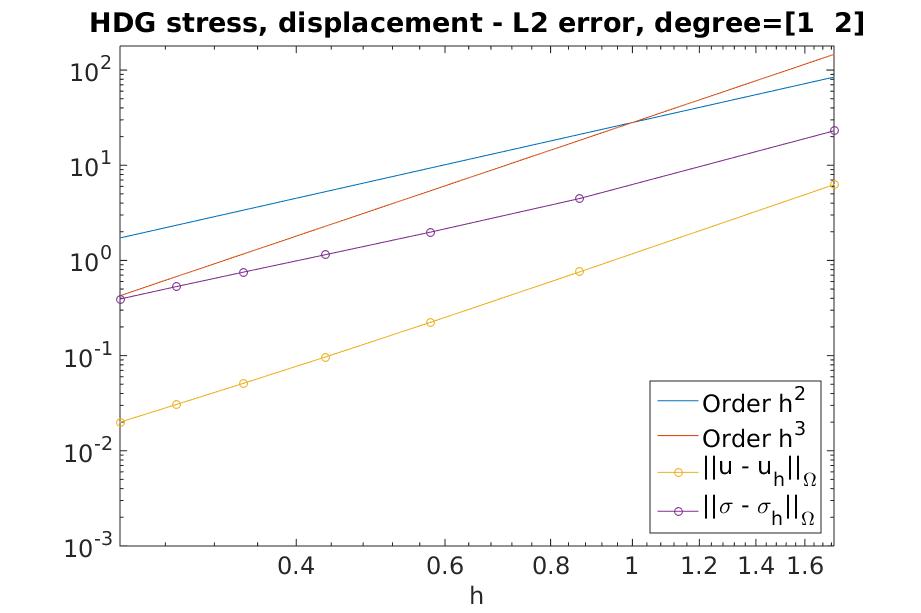}}
\subfloat{\includegraphics[width=0.5\textwidth]{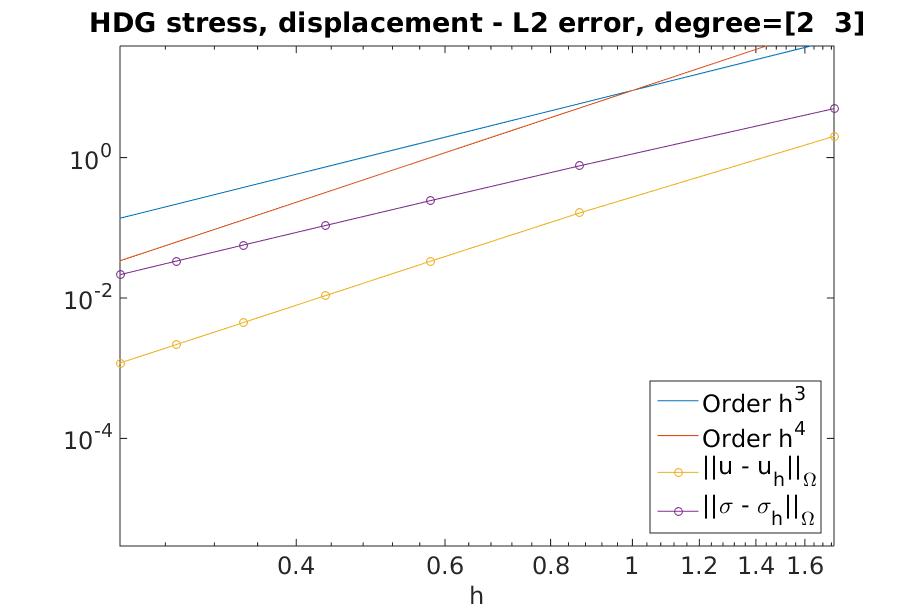}}\\
\subfloat{\includegraphics[width=0.5\textwidth]{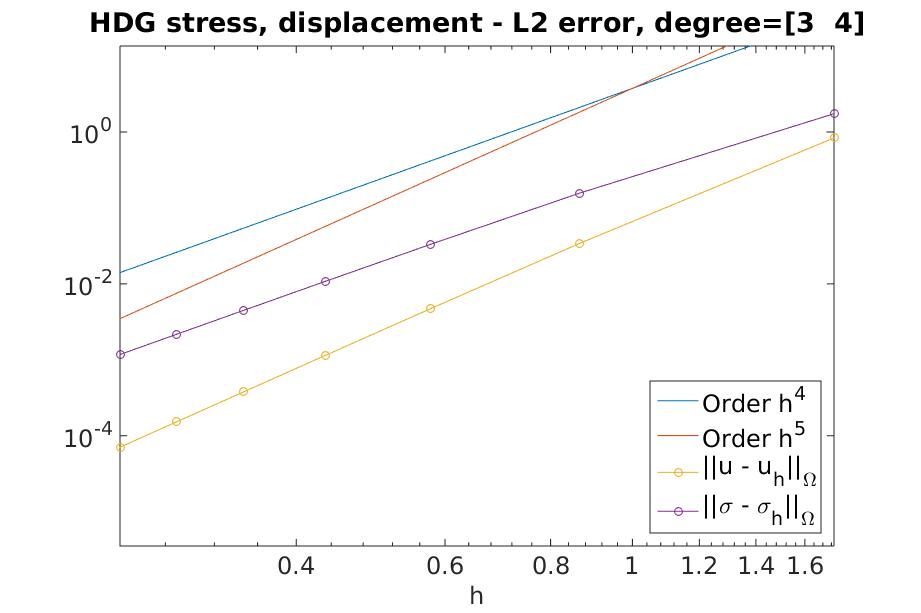}}
\subfloat{\includegraphics[width=0.5\textwidth]{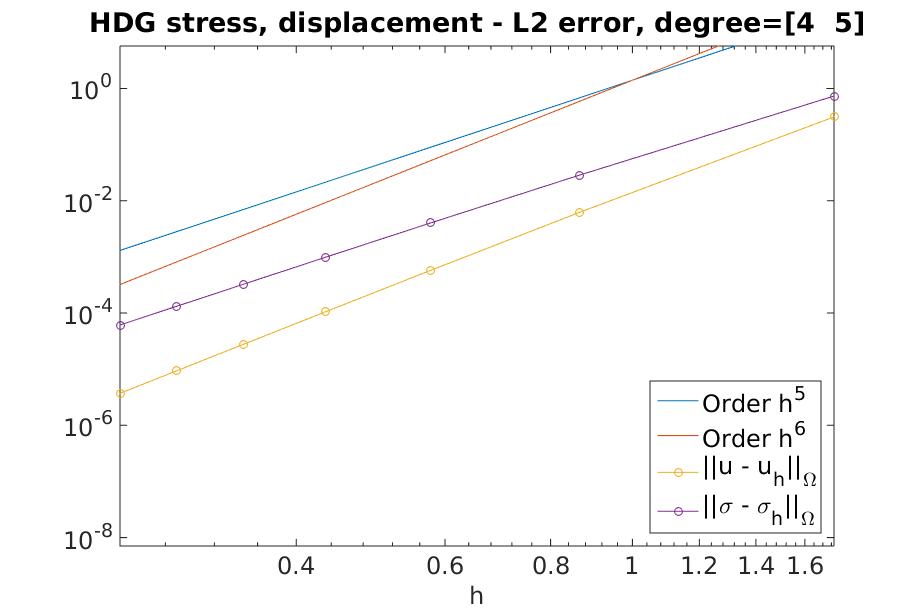}}\\
\subfloat{\includegraphics[width=0.5\textwidth]{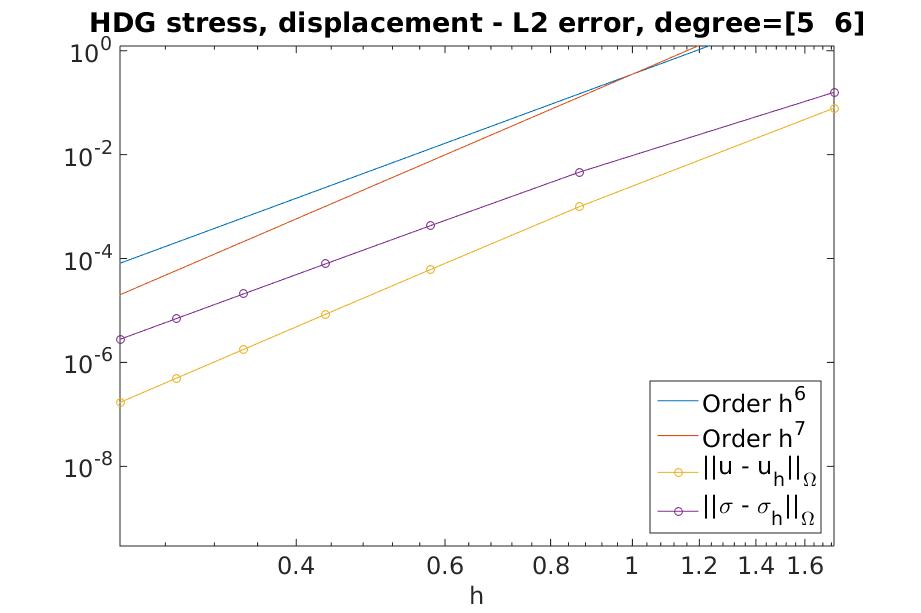}}
\subfloat{\includegraphics[width=0.5\textwidth]{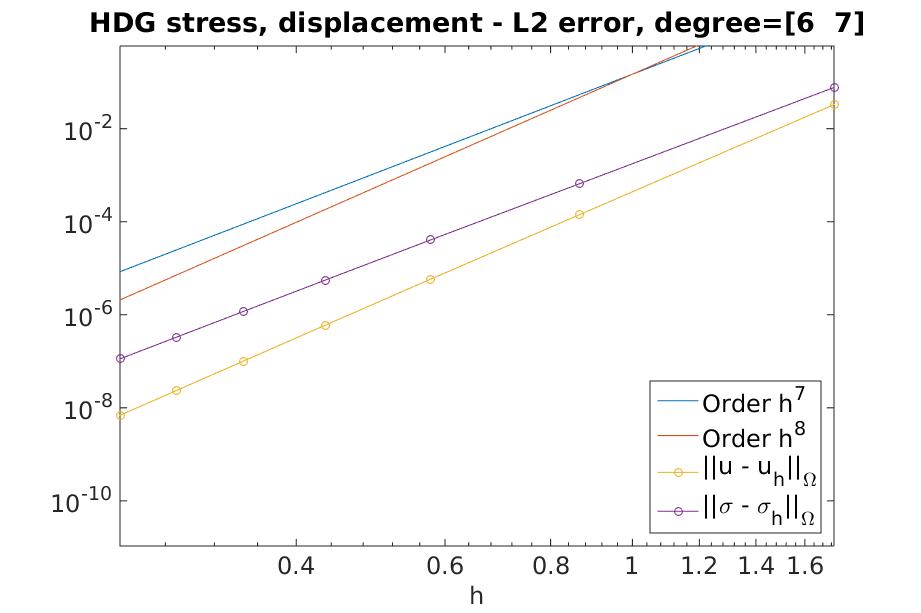}}
\caption{Errors of the first-order-in-frequency formulation ($\alpha_\kappa = i\kappa, \tau = h^{-1}$ in~\eqref{eq:7.2}) for the non-homogeneous test case.}
\label{fig:test1-fo}
\end{figure}

We next explore the method based on the pure second-order in frequency formulation ($\alpha_\kappa = 1, \tau = h^{-1}$ in~\eqref{eq:7.2}). Approximation errors for several values of the polynomial degree $k$ are shown in Figure~\ref{fig:test1-so}.

\begin{figure}
\centering
\subfloat{\includegraphics[width=0.5\textwidth]{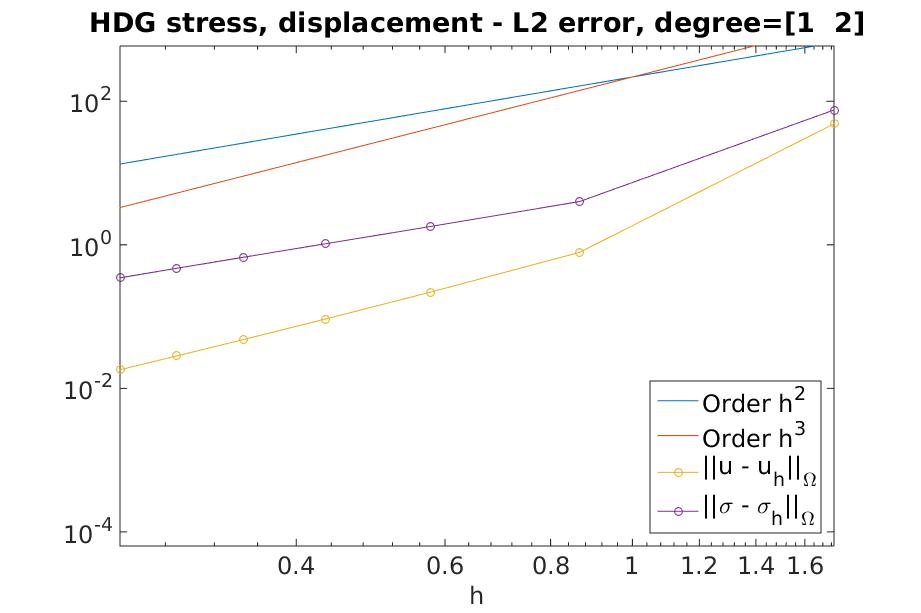}}
\subfloat{\includegraphics[width=0.5\textwidth]{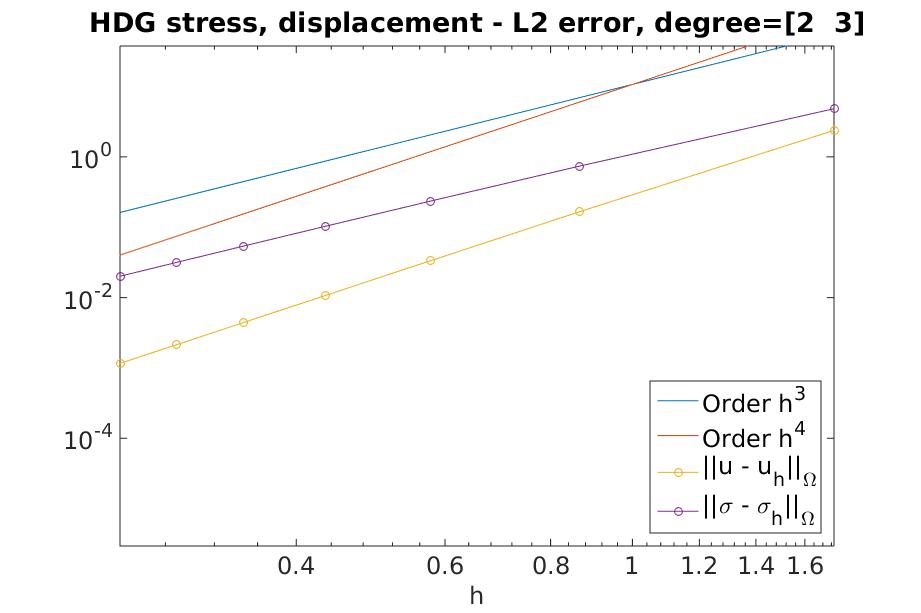}}\\
\subfloat{\includegraphics[width=0.5\textwidth]{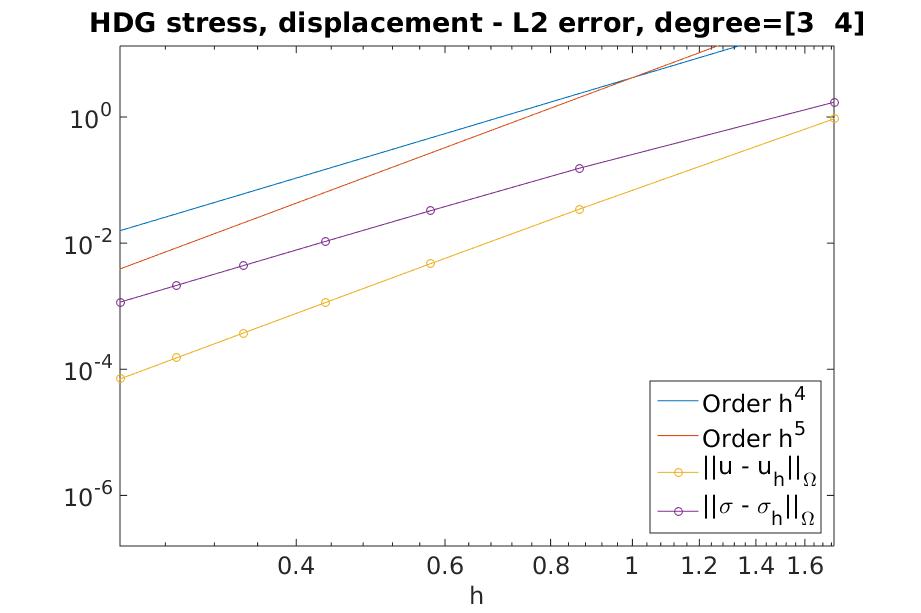}}
\subfloat{\includegraphics[width=0.5\textwidth]{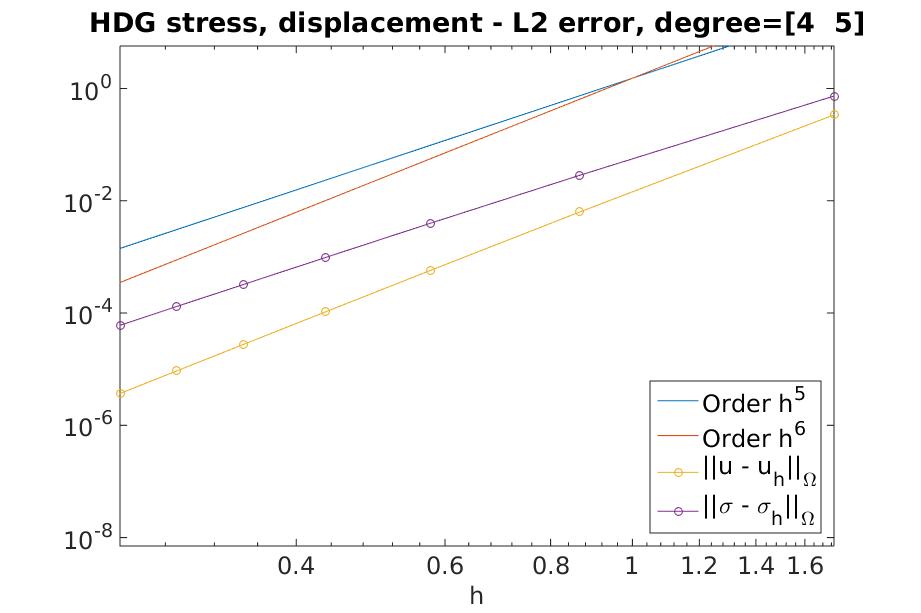}}\\
\subfloat{\includegraphics[width=0.5\textwidth]{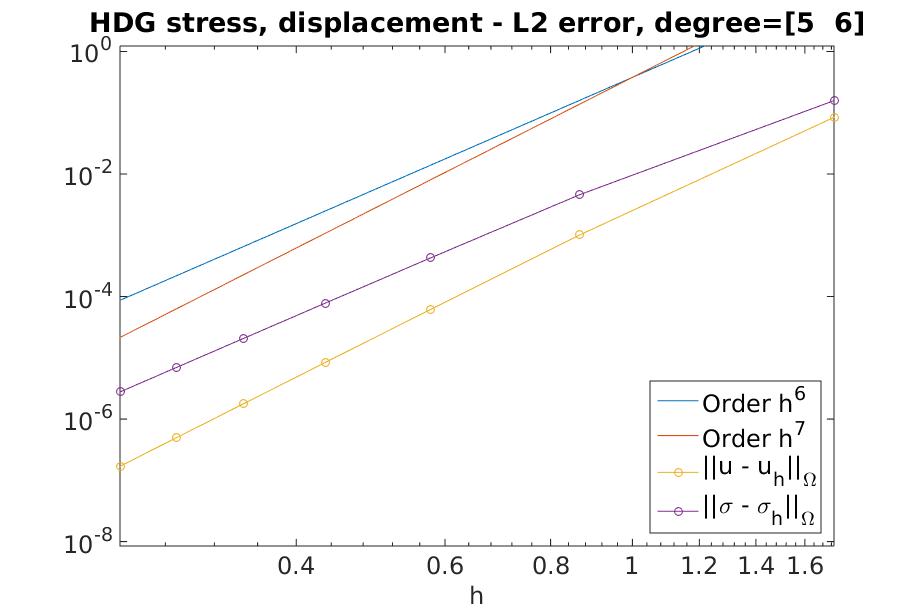}}
\subfloat{\includegraphics[width=0.5\textwidth]{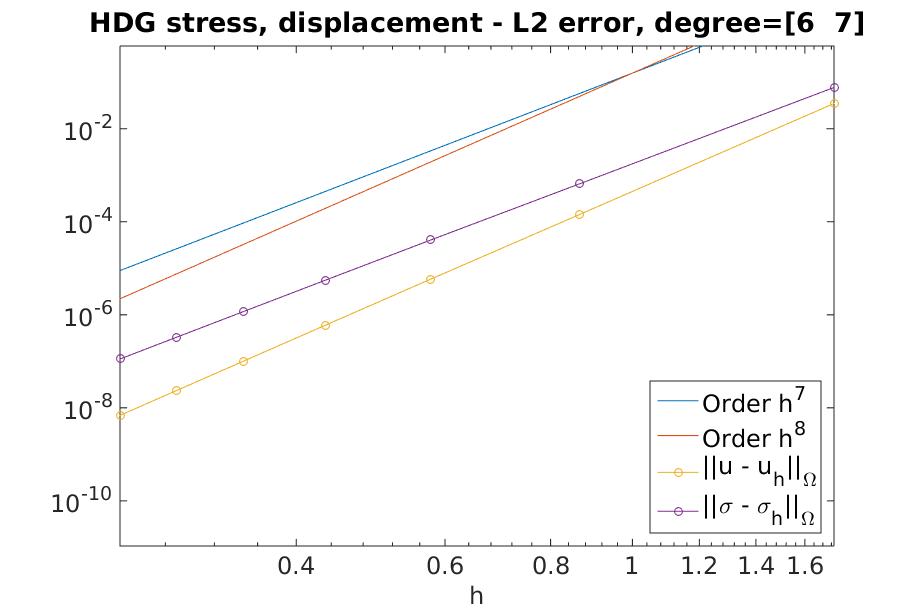}}
\caption{Errors of the second-order-in-frequency formulation ($\alpha_\kappa = 1, \tau = h^{-1}$ in~\eqref{eq:7.2}) for the non-homogeneous test case.}
\label{fig:test1-so}
\end{figure}

\paragraph{Plane wave solutions.}  \label{pwave}
We now test a problem with constant coefficients and plane-wave solutions as exact data, with Dirichlet conditions imposed on the entire boundary. Plane waves for the time-harmonic elastic wave equation are functions of the form:
\[
\mathbf u(\mathbf x) = a\, \mathbf e \exp \left( (-\imath\kappa/c) \mathbf x \cdot \mathbf d \right),
\]
where $\mathbf d$ (direction of propagation) and $\mathbf e$ (direction of elastic displacement) are unit vectors, $a$ is a constant and: (a)
either $\mathbf d = \mathbf e$ and the wavespeed is $c = \sqrt{(\lambda + 2 \mu)/\rho}$ (pressure wave), or (b)
$\mathbf d \cdot \mathbf e = 0$ and $c = \sqrt{\mu/\rho}$
(shear wave).

We use amplitude $a=0.3$ and frequency $\kappa = 1$ and tetrahedrizations of the cube each with $6n^3$ elements.  Convergence results for pressure and shear waves with the first-order-in-frequency formulation ($\alpha_\kappa = i\kappa, \tau = h^{-1}$ in~\eqref{eq:7.2}) are shown in Figure~\ref{fig:pressurewave-1}. Both show the same pattern of optimal order $\mathcal O(h^{k+2})$ for the displacement variable and $\mathcal O(h^{k+1})$ for the stress variable.

\begin{figure}
\centering
\subfloat{\includegraphics[width=0.5\textwidth]{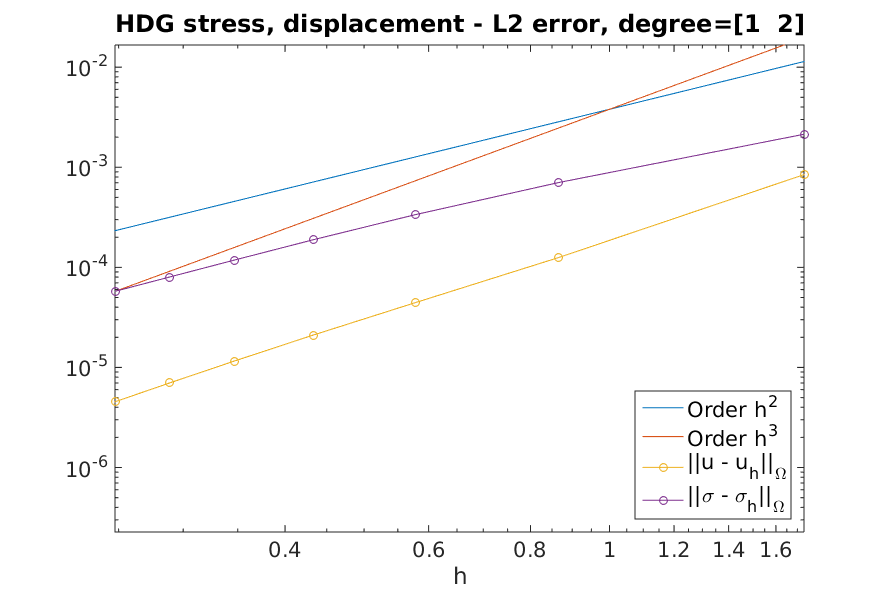}}
\subfloat{\includegraphics[width=0.5\textwidth]{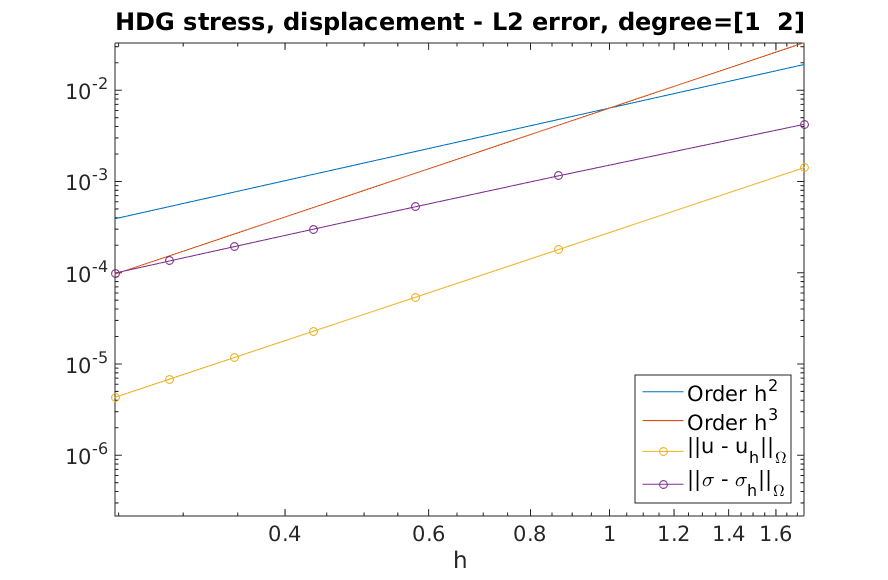}}\\
\subfloat{\includegraphics[width=0.5\textwidth]{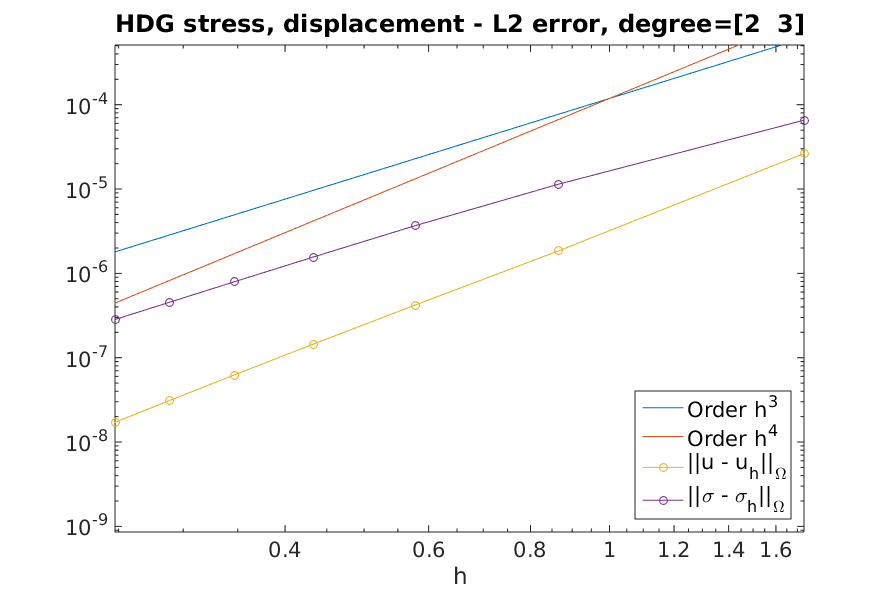}}
\subfloat{\includegraphics[width=0.5\textwidth]{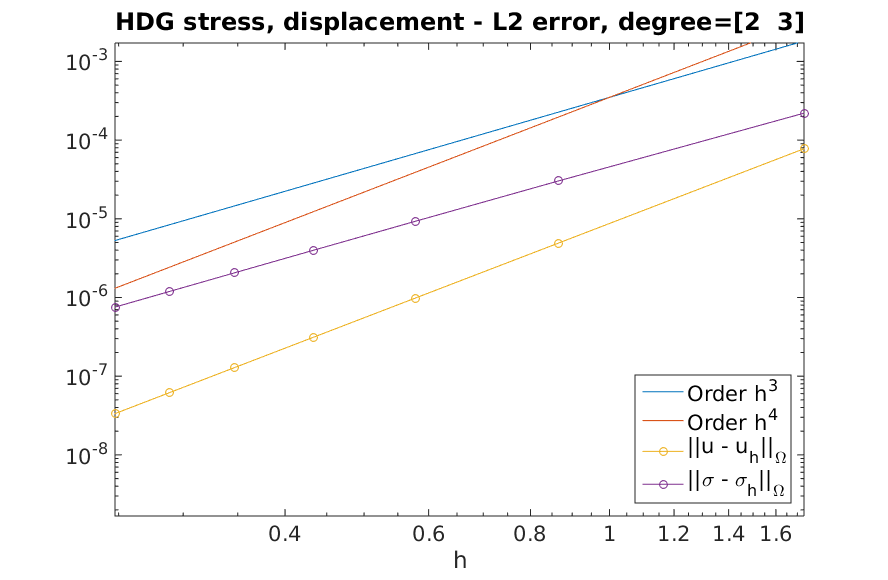}}\\
\subfloat{\includegraphics[width=0.5\textwidth]{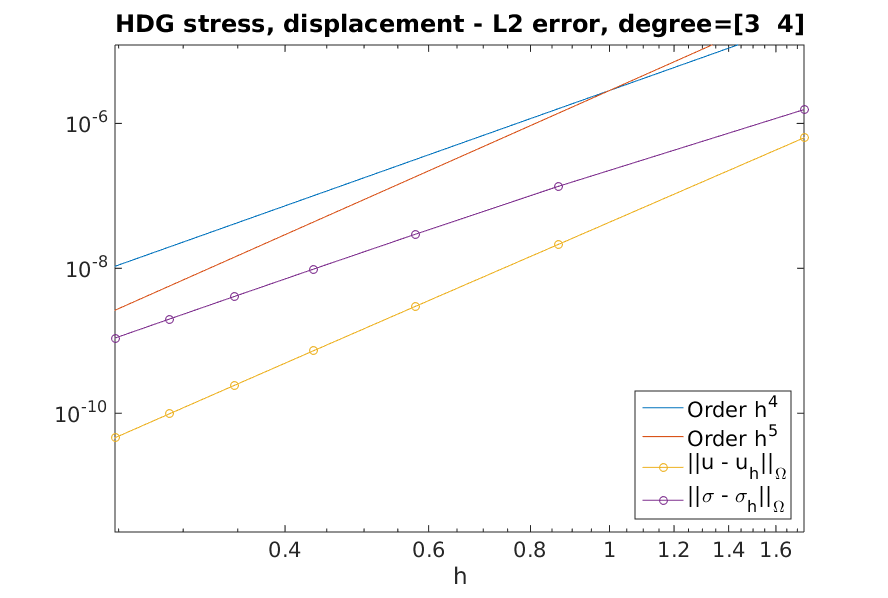}}
\subfloat{\includegraphics[width=0.5\textwidth]{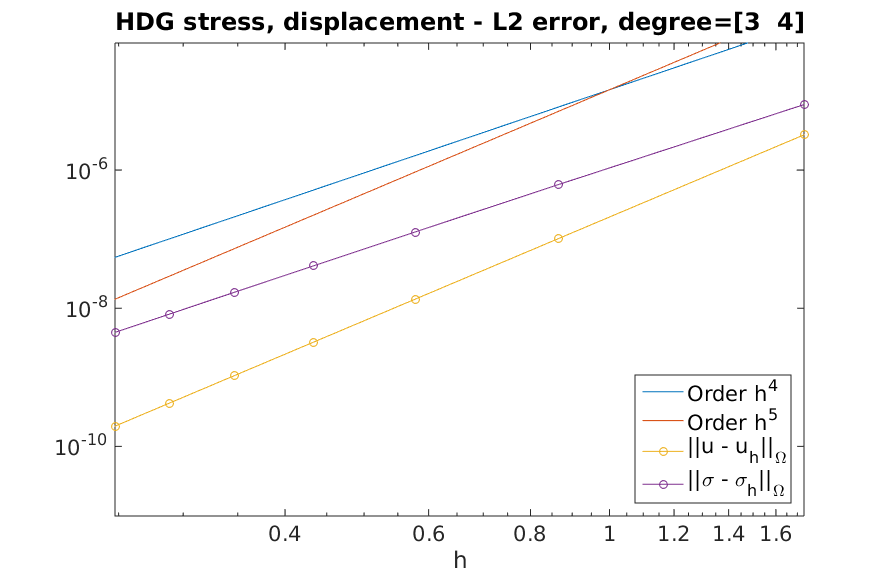}}
\caption{Convergence results for pressure (left column) and shear (right column) waves with the first-order-in-frequency formulation ($\alpha_\kappa = i\kappa, \tau = h^{-1}$ in~\eqref{eq:7.2}) increasing the polynomial degree $k$ from $1$ to $3$}
\label{fig:pressurewave-1}
\end{figure}

\paragraph{Increasing the frequency $\kappa$.}\label{sec:7.4}
In a final test, we look for the behavior of the method as the frequency increases $h\to 0$ while $h \kappa$ remains constant. We will take $\kappa$-dependent plane waves as exact solutions. In order to play with different values of $\kappa$ without hitting elastic modes, we impose an impedance boundary condition
\[
\widetilde{\bs\sigma}\,\mathbf n -\imath \kappa \mathbf u =\widetilde{\mathbf g}_R \qquad \mbox{on $\Gamma$}.
\]
Note that we have not carried out the analysis for these boundary conditions, but that convergence estimates can be obtained with uncomplicated modifications of what we have done in the previous sections. The HDG equation corresponding to the boundary conditions and interelement stress balance become
\[
\langle\widehat{\bs\sigma}_h\mathbf n - \bs\tau(\PM\mathbf u_h - \widehat{\mathbf u}_h),\bs\mu\rangle_{\partial\Th}-\imath\kappa\langle \widehat{\mathbf u}_h, \bs\mu \rangle_\Gamma
= \langle \widetilde{\mathbf g}_R, \bs\mu \rangle_\Gamma \qquad \forall \bs\mu \in \Mh.
\]
This is just a slight perturbation of the pure Neumann problem, so the linear system (with coefficient representations as in Section \ref{sec:6}) may be represented by
\[
\left[\begin{array}{ccc}
	\mathrm A & \mathrm D^\top & -\mathrm N^\top \\
	-\mathrm D & -\kappa^2\mathrm M-\imath\kappa\mathrm T_{11} & \imath\kappa\mathrm T_{12} \\
	\mathrm N & -\imath\kappa\mathrm T_{12}^\top & \imath\kappa\mathrm T_{22} -\imath\kappa \mathrm C
\end{array}\right]
\left[\begin{array}{l} 
	\widetilde{\underline\sigma} \\ \underline u\\ \underline{\widehat u}
\end{array}\right]
=
\left[\begin{array}{c}
	\underline 0 \\ \widetilde{\underline f} \\ \widetilde{\underline g}_R
\end{array}\right],
\]
where $\mathrm C$ is a boundary mass matrix related to the space $\Mh$ restricted to $\Gamma$. We are going to plot relative errors
\[
\|\mathbf u-\mathbf u_h\|_\Omega/\|\mathbf u\|_\Omega,
\qquad
\|\widetilde{\bs\sigma}-
\widetilde{\bs\sigma}_h\|_\Omega/\|\widetilde{\bs\sigma}\|_\Omega,
\]
for a given exact plane pressure wave solution. We will fix the value $h\kappa$ and use finer grids on the cube, thus automatically  increasing the frequency. Note that for plane waves $\| \mathbf u\|_\Omega=\mathcal O(1)$ but $\|\widetilde{\bs\sigma}\|_\Omega=\mathcal O(\kappa)=\mathcal O(h^{-1})$. We are expecting errors to stay bounded as $h\to 0$ (notice that we do not have a theory supporting this and we do not go very far with the frequency). Some experiments are reported in Figure \ref{fig:hk-constant-a}, using the lowest order method ($\mathcal P_1$ approximation for the stress, $\mathcal P_2$ for the displacement) with different values of $h\kappa$. That the error for the displacement seems to be declining as $h\to 0$, which shows the numerical method doing somewhat better than expected for the displacement. 

\begin{figure}
\centering
\subfloat[$h\kappa = \sqrt{3}/100$]{\includegraphics[width=\textwidth]{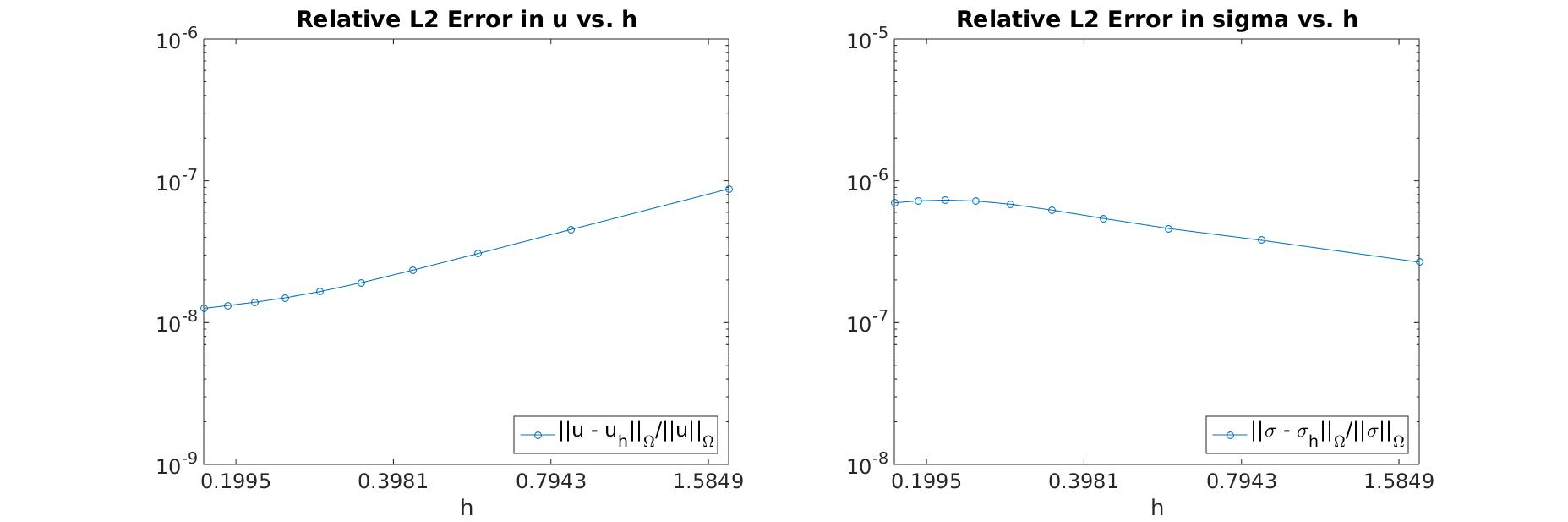}}\\
\subfloat[$h\kappa = \sqrt{3}/10$]{\includegraphics[width=\textwidth]{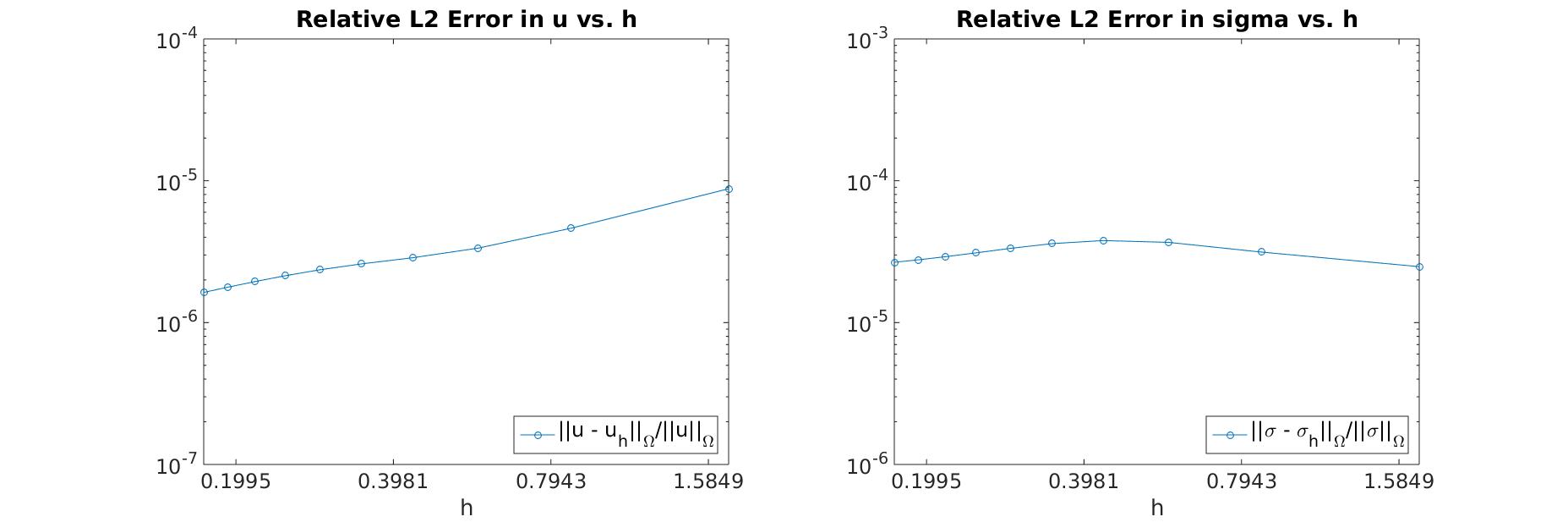}}\\
\subfloat[$h\kappa = \sqrt{3}$]{\includegraphics[width=\textwidth]{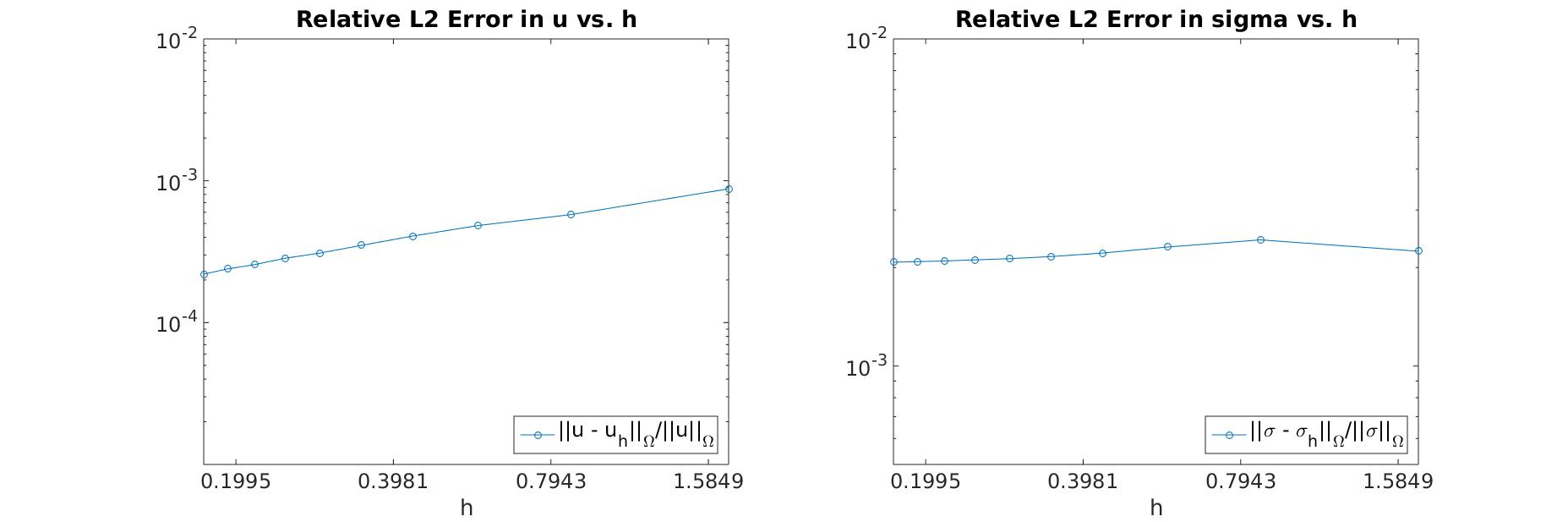}}\\
\caption{Convergence results for pressure waves with the first-order-in-frequency formulation ($\alpha_\kappa = i\kappa, \tau = h^{-1}$ in~\eqref{eq:7.2}), using
$\mathbf u_h|_K\in P_2(K;\mathbb C^3), \sigma_h|_K\in P_1(K;\mathbb C_\text{sym}^{3\times3})$ and keeping $h\kappa$ constant.}
\label{fig:hk-constant-a}
\end{figure}

\paragraph{Some conclusions.} We have introduced and analyzed a family of Hybridizable Discontinuous Galerkin methods for the equations of time-harmonic linear elastodynamics in three dimensions. Convergence orders are shown to be optimal in both variables (displacement and stress), even if different polynomial degrees are used for their approximation, which means that the lower approximation order of the stress does not pollute the approximation for the displacement. We have transferred the methods to an HDG semidiscretization in space of the transient model equations and shown that one of the choices leads to a conservative method. The experiments show that the methods behave well for fixed moderate frequences and that for increasing  frequencies the scheme has good asymptotic properties.

\appendix 
\section{Some additional proofs}\label{sec:A}

\begin{proof}[Proof of \eqref{eq:6.2}]
The dual adjoint problem \eqref{eq:6.1} is equivalent to
\begin{subequations}\label{eq:Y.1}
\begin{alignat}{6}
\label{eq:Y.1a}
\nabla\cdot\mathcal C\bs\varepsilon(\bs\phi) + \kappa^2\rho\,\bs\phi
	 &= \imath\kappa \eu &\qquad &\mbox{in $\Omega$},\\
\bs\phi &=\bs 0 & & \mbox{on $\Gamma_D$},\\
\mathcal C\bs\varepsilon(\bs\phi)\,\mathbf n &=\bs 0 & & \mbox{on $\Gamma_N$},
\end{alignat}
\end{subequations}
and therefore $\|\bs\phi\|_{1,\Omega} \le \kappa C_\kappa \| \eu\|_\Omega$. From a variational formulation of \eqref{eq:Y.1} we can prove the identity
\[
\| \kappa \bs\phi\|_\rho^2=\imath(\kappa \eu, \overline{\bs\phi})_\Omega+
(\mathcal C\bs\varepsilon(\bs\phi),\bs\varepsilon(\overline{\bs\phi}))_\Omega, 
\]
and therefore, using Young's inequality and hiding constants only related to the equation's coefficients, we can prove that
\[
\|\kappa \bs\phi\|_\Omega \lesssim (1+C_\kappa\kappa) \| \eu\|_\Omega.
\]
We now rewrite \eqref{eq:Y.1a} as
\[
\nabla\cdot\mathcal C\bs\varepsilon(\bs\phi)-\rho\bs\phi=\imath\kappa\eu-(\kappa^2+1)\rho\bs\phi
\]
and use the regularity estimate \eqref{eq:3.2} to bound $\|\bs\phi\|_{2,\Omega}$.
\end{proof}

\begin{proof}[Sketch of the proof of Theorem \ref{the:7.2}]
The first order in space, second order in frequency system is
\begin{subequations}
\begin{alignat*}{6}
\mathcal A\widetilde{\bs\sigma} -\bs\varepsilon(\bff u) &=\mathbf 0
	&\qquad &\mbox{in $\Omega$},\\
\nabla\cdot\widetilde{\bs\sigma} +\kappa^2\,\rho\,\bff u &=\widetilde{\bff f}
	&\qquad &\mbox{in $\Omega$},\\
\bff u &=\mathbf g_D
	&\qquad &\mbox{on $\Gamma_D$},\\
\widetilde{\bs\sigma}\mathbf n &=\widetilde{\mathbf g}_N
	&\qquad &\mbox{on $\Gamma_N$}.
\end{alignat*}
\end{subequations}
From this moment on, we will drop all tildes in the formulas. It has to be understood though that the stress that we are computing with this method is the physical stress and not the one scaled by $\imath/\kappa$. The error equations are
\begin{subequations}\label{eq:X.3}
\begin{alignat}{6}
 (\mathcal A \esig,\bs\xi)_\Th+(\eu,\nabla\cdot\bs\xi)_\Th-\langle \euhat,\bs\xi\mathbf n\rangle_{\partial\Th}
	&= (\mathcal A \epsig,\bs\xi)_\Th,\\
\nonumber
-(\nabla\cdot\esig,\mathbf w)_\Th-\kappa^2 (\rho\eu,\mathbf w)_\Th\\
	+\langle\bs\tau(\PM \eu-\euhat),\PM\mathbf w\rangle_{\partial\Th} 
	&=-\kappa^2(\rho\epu,\mathbf w)_\Th\\
\nonumber
	&\phantom{=}-\langle \epsig\mathbf n,\mathbf w\rangle_{\partial\Th}
		+\langle\bs\tau\epu,\PM\mathbf w\rangle_{\partial\Th},\\
\langle \esig\mathbf n-\bs\tau(\PM\eu-\euhat),\bs\mu\rangle_{\partial\Th\setminus\Gamma_D}
	&=\langle\epsig\mathbf n,\bs\mu\rangle_{\partial\Th\setminus\Gamma_D}
		-\langle\bs\tau \PM\epu,\bs\mu\rangle_{\partial\Th\setminus\Gamma_D}\\
\langle\euhat,\bs\mu\rangle_{\Gamma_D} &=0,
\end{alignat}
\end{subequations}
and a simple argument shows the new energy identity
\begin{alignat}{6}
\label{eq:X.0}
& \|\esig\|_{\mathcal A}^2-\kappa^2\|\eu\|_\rho^2+\|\PM \eu-\euhat\|_\tau^2 \\
& \hspace{5pt} = (\mathcal A\epsig,\overline\esig)_\Th-\kappa^2(\rho \overline\epu,\eu)_\Th
	-\langle\overline\epsig\mathbf n,\eu-\euhat\rangle_{\partial\Th}
	+\langle\bs\tau \, \overline\epu, \PM \eu-\euhat\rangle_{\partial\Th}.
	\nonumber
\end{alignat}
The adjoint problem has to be written as 
\begin{subequations}\label{eq:X.1}
\begin{alignat}{6}
\mathcal A\bs\psi +\bs\varepsilon(\bs\phi) &=\mathbf 0
	&\qquad &\mbox{in $\Omega$},\\
-\nabla\cdot\bs\psi +\kappa^2\,\rho\,\bs\phi &= -\eu
	&\qquad &\mbox{in $\Omega$},\\
\bs\phi &=\mathbf 0
	&\qquad &\mbox{on $\Gamma_D$},\\
\bs\psi\mathbf n &=\mathbf 0
	&\qquad &\mbox{on $\Gamma_N$},
\end{alignat}
\end{subequations}
where the negative sign in the right hand side is added for convenience. With the usual regularity hypotheses, the scaled regularity inequalities for the solution of this problem are:
\[
\kappa(\| \bs\phi\|_{1,\Omega} + \| \rho\bs\phi\|_{1,\Th})+\| \mathcal A \bs\psi\|_{1,\Omega}+ \|\bs\psi\|_{1,\Th}
+ \|\bs\phi\|_{2,\Omega} \le E_\kappa\|\eu\|_\Omega,
\]
with $E_\kappa$ bounded as in the statement of the theorem. The proof of this inequality is very similar to the proof of \eqref{eq:6.2} above.
After integration by parts and introduction of projections it can be shown that the solution of \eqref{eq:X.1} satisfies the following identities
\begin{subequations}\label{eq:X.4}
\begin{alignat}{6}
(\mathcal A\esig,\overline{\bs\psi})_\Th
	-(\nabla\cdot\esig,\PW\overline{\bs\phi})_\Th+\langle\esig\mathbf n,\overline{\bs\phi}\rangle_{\partial\Th} &=0,\\
(\eu,\nabla\cdot\PV\overline{\bs\psi})_\Th+\langle \eu,(\PV\overline{\bs\psi}-\overline{\bs\psi})\mathbf n\rangle_{\partial\Th}
	-\kappa^2 (\rho\,\eu,\overline{\bs\phi})_\Th & = \| \eu\|_\Omega^2,\\
\langle \euhat,\overline{\bs\psi}\mathbf n\rangle_{\partial\Th} &=0.
\end{alignat}
\end{subequations}
Testing the error equations \eqref{eq:X.3} with the conjugates of the projections of the adjoint problem and rearranging terms, we prove
\begin{subequations}\label{eq:X.5}
\begin{alignat}{6}
 (\mathcal A \esig,\overline{\bs\psi})_\Th+(\eu,\nabla\cdot \PV\overline{\bs\psi})_\Th
	-\langle \euhat,\overline{\bs\psi}\mathbf n\rangle_{\partial\Th}
	&=\ell_1(\bs\psi),\\
-(\nabla\cdot\esig,\PW\overline{\bs\phi})_\Th-\kappa^2 (\rho\eu,\overline{\bs\phi})_\Th
	&=\ell_2(\bs\phi),\\
\langle \esig\mathbf n,\overline{\bs\phi}\rangle_{\partial\Th}
	&=\ell_3(\bs\phi),
\end{alignat}
\end{subequations}
where
\begin{eqnarray*}
\ell_1(\bs\psi)&:=&
	(\mathcal A\esig,\overline{\bs\psi}-\PV\overline{\bs\psi})_\Th
	+(\mathcal A \epsig,\PV\overline{\bs\psi})_\Th
	-\langle\euhat,(\overline{\bs\psi}-\PV\overline{\bs\psi})\mathbf n\rangle_{\partial\Th},\\
\ell_2(\bs\phi) &:=& -\kappa^2\left(
	(\rho\eu,\overline{\bs\phi}-\PW\overline{\bs\phi})_\Th
	+(\rho\epu,\PW\overline{\bs\phi})_\Th\right)\\
	 & & -\langle \epsig\mathbf n,\PW\overline{\bs\phi}\rangle_{\partial\Th}
		+\langle\bs\tau\PM\epu,\PW\overline{\bs\phi}\rangle_{\partial\Th}
		 -\langle\bs\tau(\PM \eu-\euhat),\PW\overline{\bs\phi}\rangle_{\partial\Th} \\
\ell_3(\bs\phi) &:=& \langle\epsig\mathbf n,\PM\overline{\bs\phi}\rangle_{\partial\Th}
		-\langle\bs\tau \PM\epu,\overline{\bs\phi}\rangle_{\partial\Th}
		+\langle \bs\tau(\PM\eu-\euhat),\overline{\bs\phi}\rangle_{\partial\Th}.
\end{eqnarray*}
The sum of equations \eqref{eq:X.4} can then be compared with the sum of equations \eqref{eq:X.5} to prove the duality identity:
\begin{eqnarray}\label{eq:X.6}
\| \eu\|_\Omega^2
	&=& ( (\mathcal A\esig,\overline{\bs\psi}-\PV\overline{\bs\psi})_\Th+(\mathcal A\epsig,\PV\overline{\bs\psi})_\Th\\
	\nonumber
	& & -\kappa^2 \left((\rho\,\eu,\overline{\bs\phi}-\PW\overline{\bs\phi})_\Th+(\rho\,\epu,\PW\overline{\bs\phi})_\Th\right)\\
	\nonumber
	& & +\langle \eu-\euhat,(\overline{\bs\psi}-\PV\overline{\bs\psi})\mathbf n\rangle_{\partial\Th}\\
	\nonumber
	& & +\langle\bs\tau(\PM \eu-\euhat)-\bs\tau\PM\epu,\overline{\bs\phi}-\PW\overline{\bs\phi}\rangle_{\partial\Th}
	       - \langle \epsig\mathbf n,\PW\overline{\bs\phi}-\PM\overline{\bs\phi}\rangle_{\partial\Th}.
\end{eqnarray}
What is left now is the proof of bounds for the right-hand sides of \eqref{eq:X.0} and \eqref{eq:X.6}. This process requires just going carefully over the proofs of Proposition \eqref{prop:5.4} and \eqref{prop:6.2}. Nothing essential is changed. We can write the results with our shorthand notation for errors $\Sigma:=\| \esig\|_{\mathcal A}$, $\mathrm T:=\| \PM \eu-\euhat\|_\tau$, $\mathrm U:=\| \eu\|_\Omega$, and approximation terms $\Sigma_h:= h^t |\bs\sigma|_{t,\Omega}$, $\mathrm U_h:=h^{s-1} |\mathbf u|_{s,\Omega}$.
The bounds we obtain are:
\begin{align*}
\Sigma^2+ \mathrm T^2
	& \lesssim \Sigma\,\Sigma_h + \kappa^2 \mathrm U\,\mathrm U_h 
	+ (\Sigma_h+\mathrm U_h) \mathrm T +\Sigma_h^2+\kappa^2\mathrm U^2, \\
\mathrm U & \lesssim \alpha^2 (\Sigma+\Sigma_h+\mathrm U_h +\mathrm T),
\end{align*}
where $\alpha:=E_\kappa h (1+\kappa)$.
The condition that allows us to bootstrap is $C(\alpha\kappa)^2 \le 1/4$, where $C$ is a constant related to the constants hidden in the symbols $\lesssim$ above. After simplification, we prove
\[
\Sigma+\mathrm T  \lesssim \Sigma_h+(1+\kappa) \mathrm U_h \qquad 
\mathrm U  \lesssim \alpha (\Sigma_h+(1+\kappa)\mathrm U_h),
\]
which is the statement of the theorem.
\end{proof}

%
\bibliography{HDGpaperBibliography}
%

\end{document}